\pdfoutput=1
\documentclass{article}
\usepackage{arxiv}
\usepackage[toc,page]{appendix}

\usepackage[usenames]{xcolor}

\usepackage{inputenc} %
\usepackage[T1]{fontenc}    %
\usepackage{hyperref}       %
\usepackage{url}            %
\usepackage{booktabs}       %
\usepackage{amsfonts}       %
\usepackage{nicefrac}       %
\usepackage{microtype}      %
\usepackage{lipsum}
\usepackage{mathtools}
\usepackage{cite}
\usepackage{amsmath}
\usepackage{amssymb}
\usepackage{amsfonts}
\usepackage{amsthm}
\usepackage{algorithm,algcompatible}
\usepackage[toc,page]{appendix}

\usepackage{lineno}

\newtheorem{theorem}{Theorem}
\newtheorem{lemma}{Lemma}

\newtheorem{proposition}{Proposition}

\numberwithin{equation}{section}

\begin{document}

\title{Nearly optimal bounds on the Fourier sampling numbers of Besov spaces
}

\author{Jonathan W. Siegel \\
 Department of Mathematics\\
 Texas A\&M University\\
 College Station, TX 77843 \\
 \texttt{jwsiegel@tamu.edu} \\
}

\maketitle

\begin{abstract}
    Let $\mathbb{T}^d$ denote the $d$-dimensional torus. We consider the problem of optimally recovering a target function $f^*:\mathbb{T}^d\rightarrow \mathbb{C}$ from samples of its Fourier coefficients. We make classical smoothness assumptions on $f^*$, specifically that $f^*$ lies in a Besov space $B^s_\infty(L_q)$ with $s > 0$ and $1\leq q\leq \infty$, and measure recovery error in the $L_p$-norm with $1\leq p\leq \infty$. Abstractly, the optimal recovery error is characterized by a `restricted' version of the Gelfand widths, which we call the Fourier sampling numbers. Up to logarithmic factors, we determine the correct asymptotics of the Fourier sampling numbers in the regime $s/d > 1 - 1/p$. We also give a description of nearly optimal Fourier measurements and recovery algorithms in each of these cases. In the other direction, we prove a novel lower bound showing that there is an asymptotic gap between the Fourier sampling numbers and the Gelfand widths when $q = 1$ and $p_0 < p\leq 2$ with $p_0 \approx 1.535$. Finally, we discuss the practical implications of our results, which imply a sharper recovery of edges, and provide numerical results demonstrating this phenomenon.\\\\
    \textbf{Keywords}: compressive sensing, $n$-widths, sampling numbers, Fourier measurements\\
    \textbf{Mathematics Subject Classification}: 41A46, 42A05, 46B09, 46N40, 65T40
\end{abstract}

\section{Introduction}
A fundamental problem in approximation theory, compressive sensing, and functional analysis is the following: Suppose that we are tasked with identifying an unknown function $f^*$. We obtain information about $f^*$ by observing the values of a finite set of measurements $\lambda_i(f^*)$ for $i=1,...,n$. What is the optimal set of measurements to choose and how can we optimally recover $f^*$ from these measurements?

In order to properly formulate this problem, we need to specify three pieces of information. Namely, which assumptions are made on the target function $f^*$, which measure of error will we use to evaluate our estimate of $f^*$, and what types of measurements $\lambda_i$ are allowed. By making different choices for these three problem parameters, we arrive at a variety of interesting problems, many of which serve as useful models of practical imaging and measurements systems.

Classically, a standard assumption on the function $f^*$ is that it lies in the unit ball of a Sobolev or Besov space. To simplify the presentation, we will only consider Besov spaces and restrict our attention to the case where the domain of $f^*$ is the $d$-dimensional torus $\mathbb{T}^d:= (\mathbb{R}/2\pi \mathbb{Z})^d$. All $L_p$ and $L_q$ norms will be taken over $\mathbb{T}^d$, normalized to have unit measure, i.e.,
\begin{equation}
    \int_{\mathbb{T}^d} := \frac{1}{(2\pi)^d}\int_{[0,2\pi]^d}.
\end{equation}

For parameters $1\leq q,r\leq \infty$ and index $s > 0$, we define the Besov semi-norm of a function $f:\mathbb{T}^d\rightarrow \mathbb{C}$ as
\begin{equation}\label{besov-semi-nor-definition}
    |f|_{B^s_r(L_q)} := \begin{cases}
        \left(\int_{0}^\infty [t^{-s}\omega_k(f,t)_q]^r\frac{dt}{t}\right)^{1/r} & r < \infty\\
        \sup_{t > 0} t^{-s}\omega_k(f,t)_q & r = \infty.
    \end{cases}
\end{equation}
Here the $k$-th order modulus of smoothness in $L_q$ is defined by
\begin{equation}
    \omega_k(f,t)_q := \sup_{|h|\leq t} \|\Delta_h^k f\|_{L_q},
\end{equation}
where $\Delta_h$ is the finite difference operator given by $\Delta_h f(x) = f(x + h) - f(x)$. It is well-known that any choice of $k$ greater than $s$ leads to an equivalent semi-norm in \eqref{besov-semi-nor-definition} (see for instance \cite{devore1993constructive}, Chapter 2). The Besov norm is then typically defined by
\begin{equation}
    \|f\|_{B^s_r(L_q)} := \|f\|_{L_q} + |f|_{B^s_r(L_q)}.
\end{equation}
In order to simplify the presentation, we will restrict our attention to the case $r = \infty$ in the following. We will make the assumption that $f^*$ lies in the unit ball of a such a Besov space, i.e., 
\begin{equation}
    f^*\in K_q^s := \{f:\mathbb{T}^d\rightarrow \mathbb{C},~\|f\|_{B^s_\infty(L_q)} \leq 1\}.
\end{equation}
We would like to recover $f^*$, with error measured in the $L_p$-norm, from a finite set of measurements $m_i := \lambda_i(f^*)$. In order for any practically implementable method to recover $f^*$ in this setting, the Besov ball $K^s_q$ must be a compact subset of $L_p$, which is guaranteed by the embedding condition
\begin{equation}\label{sobolev-embedding-condition}
    \frac{s}{d} > \left(\frac{1}{q} - \frac{1}{p}\right)_+.
\end{equation}
Next, let us consider the types of measurements which will be allowed. Practical imaging and measurements systems, such as MRI or CT, interact with the ground truth $f^*$ via a typically non-linear PDE. Despite this, in many cases the measurement process can be accurately modeled by a linear functional applied to $f^*$. For example, MRI measurements are accurately modeled by samples of the Fourier transform of the magnetization distribution \cite{twieg1983k,candes2006robust,lustig2007sparse}, and CT measurements are accurately modeled by samples of the Radon transform of the attenuation coefficient of the underlying material (see \cite{natterer2001mathematical}, Chapter 3 or \cite{kuchment2013radon}). 

Consequently, in this work we will restrict ourselves to measurements $\lambda_i$ which are linear functionals of $f^*$. If arbitrary linear functionals $\lambda_i\in L_{p^*}$ are allowed (here $1/p + 1/p^* = 1$, recall that $K^s_q$ is viewed as a subset of $L_p$), then the optimal recovery is governed by the Gelfand widths \cite{lorentz1996constructive,pinkus2012n}. We recall that for a compact set $K\subset X$ in a Banach space $X$, the Gelfand widths are defined by
\begin{equation}
    d^n(K)_X := \inf_{\lambda_1,...,\lambda_n\in X^*} \sup\{\|f\|_X:~f\in K~\text{and}~\lambda_i(f) = 0~\text{for $i=1,...,n$}\}.
\end{equation}
It is known that if $K$ is convex and centrally symmetric, then the best possible reconstruction error from $n$ linear measurements that one can obtain for the class $K$ is $d^n(K)_X$ (up to potentially a factor of $2$), and the reconstruction can be obtained by choosing any function in $K$ which satisfies the observed measurements \cite{micchelli1976optimal,traub1973theory,binev2024optimal}. A general recovery algorithm for linear measurements is described in Section \ref{numerical-experiments-section}.

Applying this to the setting $K = K^s_q\subset L_p$, we obtain that the optimal recovery error of $f^*\in K^s_q$ with respect to the $L_p$-norm is given by the Gelfand widths $d^n(K^s_q)_{L_p}$. The asymptotic decay of these widths were determined following the seminal work of Kashin \cite{kashin1977widths}. They are given by (see \cite{lorentz1996constructive}, Chapter 14 or \cite{pinkus2012n,dung2018hyperbolic})
\begin{equation}\label{gelfand-widths-asymptotics}
    d^n(K^s_q)_{L_p} \eqsim \begin{cases}
        n^{-s/d + (1/q - 1/p)_+} & q \geq 2\\
        n^{-s/d + (1/2 - 1/p)_+} & q < 2,
    \end{cases}
\end{equation}
as long as $s$ is large enough that $B^s_\infty(L_1)$ compactly embeds into $L_p$, i.e., when $s/d > 1-1/p$. In the case when $q=1$, which is of greatest interest to us, this does not require any restriction on $s$ beyond the compact embedding condition \eqref{sobolev-embedding-condition}. If $q > 1$, it is possible for \eqref{sobolev-embedding-condition} to hold when $s/d \leq 1-1/p$, and in this case the Gelfand widths are more subtle. For instance, in the boundary case when $s/d = 1-1/p$, additional logarithmic factors can appear \cite{malykhin2021kolmogorov,malykhin2025kolmogorov}. In this paper, we will restrict ourselves to the case $s/d > 1-1/p$ where the Gelfand widths behave like \eqref{gelfand-widths-asymptotics}, and leave the remaining cases as future work.

From a practical perspective, a disadvantage of the Gelfand widths is that they allow measurements which are general linear functionals. In the regime where $q \geq 2$ or $p \leq q$, nearly optimal measurements can be explicitly given, for instance the $n$ lowest Fourier modes will achieve the asymptotic bound \eqref{gelfand-widths-asymptotics}. However, in the regime $q < 2$ and $p > q$, measurements achieving the asymptotic bound \eqref{gelfand-widths-asymptotics} cannot be explicitly given, and instead can only be constructed using a probabilistic argument \cite{kashin1977widths,lorentz1996constructive}. Such measurements would be difficult to implement in any practical imaging or signal processing system. Consequently, it is of significant interest to consider `restricted' versions of the Gelfand widths, which make additional restrictions on the types of linear measurements allowed.

A well-known example of such restricted widths are the sampling numbers, for which the linear functionals $\lambda_i$ are restricted to be point evaluations. Specifically, the sampling numbers of a compact, convex, and centrally symmetric subset $K \subset C(\mathbb{T}^d)$ with respect to the $L_p$-norm is defined by
\begin{equation}
    s_n(K)_{L_p} := \inf_{x_1,...,x_n\in \mathbb{T}^d} \sup\{\|f\|_{L_p}:~f\in K~\text{and}~f(x_i) = 0~\text{for $i=1,...,n$}\}.
\end{equation}
We remark that for Besov balls, the condition $K^s_q\subset C(\mathbb{T}^d)$ requires the condition $s/d > 1/q$ (which is stronger than the embedding condition \ref{sobolev-embedding-condition}). Sampling numbers are an important topic in approximation theory, and have been studied both for unit balls of Besov spaces \cite{vybiral2007sampling,novak2006function,bonito2025convergence} and for more general convex subsets of $C(\mathbb{T}^d)$ \cite{ullrich2025sampling,krieg2021function,dolbeault2023sharp,krieg2021function2}. Notably, the sampling numbers for Besov balls are given by (see \cite{vybiral2007sampling,novak2006function,bonito2025convergence})
\begin{equation}\label{sampling-numbers-bound-equation}
    s_n(K^s_q)_{L_p} \eqsim n^{-s/d + (1/q-1/p)_+}.
\end{equation}
Observe that the sampling numbers are asymptotically much worse than the Gelfand widths in the regime where $1\leq q < 2$ and $p > q$. Thus, restricting to point samples significantly worsens the recovery error in this regime. The practical implications of this will be discussed further in Section \ref{numerical-experiments-section}.

While sampling numbers are a good model in many practical applications where point samples are available, many imaging systems involve measurements which are more complicated than point samples. The examples we have in mind are MRI and CT, which are accurately modeled by samples of the Fourier and Radon transforms, respectively. Motivated by this, we consider `restricted' variants of the Gelfand widths which correspond to these measurement types. Specifically, we define the Fourier sampling numbers for a class $K\subset L_p$ via
\begin{equation}
    s^F_n(K)_{L_p} := \inf_{\xi_1,...,\xi_n\in \mathbb{Z}^d} \sup\{\|f\|_{L_p}:~f\in K~\text{and}~\hat{f}(\xi_i) = 0~\text{for $i=1,...,n$}\},
\end{equation}
where
\begin{equation}
    \hat{f}(\xi) := \int_{\mathbb{T}} f(x)e^{-i\xi\cdot x}dx
\end{equation}
are the Fourier coefficients of $f$. This quantity measures how accurately a function $f\in K$ can be recovered from samples of its Fourier coefficients. The key question is what error can be obtained, and which frequencies $\xi_i$ should be sampled to obtain the optimal error. 

One can also formulate this problem on a bounded domain $\Omega\subset \mathbb{R}^d$, where samples of the Fourier coefficients are replaced by samples of the Fourier transform of $f$ (set equal to $0$ outside of $\Omega$). If $f$ is compactly supported on $\Omega$, then by making $f$ periodic on a grid which contains $\Omega$ we are reduced to the problem on $\mathbb{T}^d$ considered in this work. For $f$ which are defined on $\Omega$ but not compactly supported on $\Omega$, we leave the corresponding problem to future work. On a bounded domain we can also consider a corresponding problem for recovering $f$ from samples of the Radon transform, which corresponds to recovery via CT. This is another interesting open problem, which we will not consider further here.

The goal of this work is to prove the following bound on the Fourier sampling numbers of Besov balls $K^s_q$ in $L_p$. As mentioned before, we will restrict ourselves here to the setting where the Gelfand widths are given by \eqref{gelfand-widths-asymptotics}, i.e., where $s > 1-1/p$. Determining the correct asymptotics for the Fourier sampling numbers outside of this regime is left as an open problem.

\begin{theorem}\label{main-theorem}
    Suppose that $1 \leq q, p\leq \infty$ and $s/d > 1 - 1/p$. Then for $n \geq 4$ we have the following upper bounds on the Fourier sampling numbers:
    \begin{equation}\label{main-upper-bound-equation}
        s^F_n(K^s_q)_{L_p} \leq C\begin{cases}
            n^{-s/d + (1/q - 1/p)_+} & q\geq 2~\text{or}~p\leq q\\
            n^{-s/d}\log(n)^{1-1/p}\log(\log(n))^{5(1-1/p)} & 1\leq q < p \leq 2\\
            n^{-s/d + (1/2-1/p)}\sqrt{\log(n)}\log(\log(n))^{5/2}& 1\leq q < 2 < p.
        \end{cases}
    \end{equation}
    Here the constant in \eqref{main-upper-bound-equation} depends only upon $s,q,p$ and $d$, but not on $n$, and the logarithms are taken to base $2$.
\end{theorem}
Let us remark on the choice of (nearly) optimal frequencies, i.e., optimal up to a constant factor, which attain the upper bound \eqref{main-upper-bound-equation}. In the first regime, where either $q \geq 2$ or $p \leq q$, a nearly optimal choice is given by taking the lowest block of frequencies, i.e., the set
\begin{equation}
    \{\xi_i\}_{i=1}^n = \{\xi\in \mathbb{Z}^d:~|\xi|_\infty \leq m\},
\end{equation}
where $m$ is chosen such that $n \geq (2m+1)^d$.

In the other two regimes, where $1\leq q < 2$ and $p > q$, nearly optimal frequency sets are more complicated, and can be constructed as follows. Let $\alpha > 0$ such that $(d+\alpha)(1-1/p) < s$ be a parameter depending upon $s,d$ and $p$. Choose an integer $k_0\geq 1$ so that $n \eqsim 2^{k_0}$. We first take all frequencies $\xi$ such that $|\xi|_\infty \leq 2^{k_0}$. In addition, from each of the following overlapping frequency blocks:
\begin{equation}
    B_k := \{\xi\in \mathbb{Z}^d:~2^{k-1} \leq |\xi|_\infty \leq 2^{k+1}\}~\text{for $k_0 \leq k \leq k\left(1+\frac{d}{\alpha}\right)$},
\end{equation}
we randomly sample $\lceil 2^{dk_0 - \alpha(k-k_0)}\rceil$ additional frequencies. The resulting hierarchically randomly sub-sampled set of frequencies will realize the bound \eqref{main-upper-bound-equation} with high probability.

Theorem \ref{main-theorem} shows that the Fourier sampling numbers satisfy the same asymptotics as the Gelfand numbers up to logarithmic factors. An interesting and quite difficult question is whether these logarithmic factors are sharp. This problem is closely related to the problem of determining the minimal number of Fourier measurements required to satisfy the restricted isometry property (see \cite{candes2006near,rudelson2008sparse}). In this direction, we are able to obtain the following lower bound.
\begin{theorem}\label{lower-bound-theorem}
    Suppose that $s/d \geq 1 - 1/p$ and $1\leq p\leq 2$. Then we have the lower bound
    \begin{equation}
        s^F_n(K^s_1)_{L_p} \geq cn^{-s/d}\log(n)^{\gamma_p},
    \end{equation}
    where the exponent $\gamma_p$ is given by
    \begin{equation}
        \gamma_p := \max\left\{0,\frac{\log\left(\frac{\pi}{2^{1+1/p}}\right)}{\log{2}}\right\}.
    \end{equation}
    Here the constant $c > 0$ depends only on $s,d$ and $p$, but not on $f$ or $n$.  
\end{theorem}
Theorem \ref{lower-bound-theorem} is only informative when $\gamma_p > 0$, otherwise it reduces to the lower bound on the Gelfand widths implicit in \eqref{gelfand-widths-asymptotics}. This occurs for $p > 1/(\log(\pi)/\log(2) - 1) \approx 1.535$. For $p = 2$, the exponent $\gamma_2 \approx 0.151$. Thus, there is still a significant logarithmic gap with the upper bound \eqref{main-upper-bound-equation}. Nonetheless, Theorem \ref{lower-bound-theorem} shows that there is a non-asympotic gap between the Gelfand widths and the Fourier sampling numbers, i.e., that restricting to Fourier measurements results in a recovery error which is asymptotically worse than the error when allowing general linear measurements.

Finally, let us remark on the connection between our results and classical results in compressive sensing \cite{candes2006robust,donoho2006compressed}. The traditional theory of compressive sensing focuses on the recovery of sparse discrete vectors, i.e., vectors $x\in \mathbb{R}^N$ with far fewer than $N$ non-zero entries. For example, the seminal work \cite{candes2006robust} considered recovering discrete vectors from randomly subsampled discrete Fourier measurements. In contrast, the problem we consider concerns the recovery of \textit{continuous} function from \textit{continuous} measurements, which we believe more closely models practical applications of compressive sensing such as MRI and CT. Of course, we make heavy use of the tools of compressive sensing, especially bounds on the restricted isometry property of Fourier matrices \cite{candes2006near,rudelson2008sparse,guedon2008majorizing}. Our main contribution is to perform the correct \textit{numerical analysis} to use these results to solve the continuous Fourier sampling problem. Note also that the sampling strategy for the continuous problem, namely the hierarchical sub-sampling described after Theorem \ref{main-theorem}, is completely different than the sampling strategy for the sparse discrete recovery problems, where a uniform random sub-sampling is used.

Since we consider a continuous sampling problem, and do not make the typical sparsity assumptions on the ground truth, our results are necessarily formulating differently than the typical results in compressive sensing. Specifically, we show that functions with smoothness measured in $L_q$ can be recovered with error measured in $L_p$ at a certain rate. The case of particular interest is when $1 \leq q < p \leq 2$. In this case, we show that the recovery error can be measured in the stronger $L_p$-norm instead of the $L_q$-norm with only a logarithmic deterioration in the rate. As is well-known in non-linear approximation, the practical effect of this is a \textit{sharper recovery of edges}. We describe this phenomenon, and provide numerical experiments demonstrating this effect in Section \ref{numerical-experiments-section}. 

Compressive sensing has been successfully applied to MRI and MR spectroscopy \cite{larson2011fast,lustig2007sparse,lustig2008compressed,studer2012compressive,guerquin2011fast}, and various theories which describe this application have been given \cite{adcock2017breaking,donoho2006compressed,tsaig2006extensions,romberg2008imaging,candes2007sparsity,adcock2016generalized}. Empirically and theoretically it has been observed that sampling strategies similar to the hierarchical sub-sampling described after Theorem \ref{main-theorem} perform well \cite{adcock2017breaking,wang2014novel,guerquin2011fast,lustig2008compressed}. However, none of the existing work is strong enough to address the problem of determining the Fourier sampling numbers which we consider here. Specifically, existing analyses either make different and much stronger assumptions on the target function, such as exact wavelet sparsity or near-sparsity in $\ell_1$ \cite{adcock2017breaking,romberg2008imaging,adcock2016generalized}, or they assume that the signal is already separated into scales which are sparse in wavelet space, and which can be sampled from and recovered independently \cite{donoho2006compressed,tsaig2006extensions,candes2007sparsity}. We also remark that our theory provides a uniform recovery guarantee, i.e., a single set of measurements will work well for all target functions in a Besov space.

The paper is organized as follows. In Section \ref{upper-bounds-section} we prove the upper bound Theorem \ref{main-theorem}. In Section \ref{lower-bounds-section} we prove the lower bound Theorem \ref{lower-bound-theorem}. In Section \ref{numerical-experiments-section} we describe a general recovery algorithm for linear measurements, discuss the practical implications of our analysis, and give numerical experiments demonstrating the recovery of functions of bounded variation from hierarchically randomly sub-sampled Fourier coefficients. Finally, in Section \ref{conclusion-section} we give some concluding remarks.

\section{Upper bounds}\label{upper-bounds-section}
In this section, we prove the upper bound \eqref{main-upper-bound-equation} in Theorem \ref{main-theorem}. Two important tools in our analysis will be the properties of the de la Vall\'ee Poussin kernel and the characterization of Besov spaces in terms of the error of approximation by trigonometric polynomials. Let us begin by summarizing the main facts we will need. 

We let
\begin{equation}
    \mathcal{T}_m^d := \left\{\sum_{\substack{\xi\in \mathbb{Z}^d\\|\xi|_\infty \leq m}}a_\xi e^{i\xi\cdot x}:~a_\xi\in \mathbb{C}\right\}
\end{equation}
denote the set of trigonometric polynomials of coordinate-wise degree at most $m$, and write
\begin{equation}\label{polynomial-error-approximation-bound}
    E_m(f)_p := \inf_{p\in \mathcal{T}_{m}^d} \|f - p\|_{L_p}
\end{equation}
for the error of approximation by trigonometric polynomials of degree less than $m$ with respect to the $L_p$-norm. The optimal trigonometric polynomial in \eqref{polynomial-error-approximation-bound} is typically not a linear function of $f$ (unless $p = 2$ when it is the partial sum of the Fourier series of $f$). Nevertheless, there are linear functions of $f$ which are near optimal in a certain sense. For example, suppose that $m$ is even and consider the (tensor product) de la Vall\'ee Poussin sums of $f$, defined via a Fourier multiplier by
\begin{equation}\label{de-la-vallee-poisson-definition-equation}
    V_mf(x) := \sum_{\xi\in \mathbb{Z}^d}\widehat{V}_m(\xi)\hat{f}(\xi) e^{i\xi\cdot x},
\end{equation}
where $\widehat{V}_m(\xi) := \prod_{j=1}^d \nu_m(\xi_j)$ and the one-dimensional multipliers $\nu_m(k)$ are given by
\begin{equation}
    \nu_m(k) := \begin{cases}
        1 & |k| \leq m/2\\
        2\left(1-\frac{|k|}{m+1}\right) & m/2 < |k| \leq m\\
        0 & |k| > m.
    \end{cases}
\end{equation}
Note that $\widehat{V}_m(\xi) = 0$ if $|\xi|_\infty > m$, and $\widehat{V}_m(\xi) = 1$ if $|\xi|_\infty \leq m/2$. From this it follows that that $V_mf\in \mathcal{T}_{m}^d$ for all $f$ and that $V_mf = f$ for any $f\in \mathcal{T}^d_{m/2}$. 

It is also well-known that the kernel $V_m$ corresponding to the Fourier multiplier $\widehat{V}_m$ is bounded in $L_1$ (uniformly in $m$), and thus the operators $V_m$ are uniformly operator on $L_p$ for $1\leq p\leq \infty$ (in fact on all function spaces with a translation invariant norm). Thus, the Lebesgue lemma implies that
\begin{equation}\label{de-la-valee-bound-1-equation}
    \|f - V_mf\|_{L_p} \leq C\inf_{p\in\mathcal{T}_{m/2}^d} \|f - p\|_{L_p} \leq CE_{m/2}(f)_p,
\end{equation}
where the constant $C$ depends only on $d$ and not on $m$ or $f$.

Next, we will need a characterization of the Besov semi-norms defined in \eqref{besov-semi-nor-definition} in terms of the error of approximation by trigonometric polynomials. Specifically, it is well-known that the Besov semi-norm is equivalent to (see for instance \cite{devore1993constructive,dung2018hyperbolic})
\begin{equation}\label{besov-space-characterization}
    |f|_{B^s_\infty(L_q)} \eqsim \sup_{m\geq 0} (m+1)^sE_{m}(f)_q \eqsim \sup_{k \geq 0}2^{ks}E_{2^k-1}(f)_q.
\end{equation}

\subsection{The linear regime: $p\leq q$ or $q \geq 2$}
Let us return to the proof of the upper bound in Theorem \ref{main-theorem}. We begin by considering the first case where $p\leq q$ or $q \geq 2$. Note that it suffices to consider the case where $n = (2m+1)^d$ for an even integer $m \geq 2$. Consider sampling Fourier coefficients on a uniform $\{-m,...,m\}^d$ grid, i.e.,
\begin{equation}
    \{\xi_i\}_{i=1}^n = \{-m,-(m-1),...,(m-1),m\}^d.
\end{equation}
Suppose that $f\in K^s_q$, i.e., $\|f\|_{B^s_\infty(L_q)} \leq 1$, and $\hat{f}(\xi_i) = 0$ for $i=1,...,n$. It follows from this and the fact that $\widehat{V}_m(\xi) = 0$ for $|\xi|_\infty > m$ that $V_m f = 0$, and thus by \eqref{de-la-valee-bound-1-equation} and \eqref{besov-space-characterization} we get
\begin{equation}
    \|f\|_{L_q} = \|f - V_mf\|_{L_q} \leq CE_{m/2}(f)_p \leq C(m/2 + 1)^{-s} \leq Cn^{-s/d}.
\end{equation}
This proves the desired bound when $p = q$, and thus also when $p < q$ by H\"older's inequality. To handle the case where $p > q$, we use the standard Besov space embedding
\begin{equation}
    \|f\|_{B^{s-\frac{d}{q} + \frac{d}{p}}_\infty(L_p)} \leq C\|f\|_{B^s_\infty(L_q)},
\end{equation}
which is valid for $p > q$ and $s - d/q+d/p > 0$. This gives the bound
\begin{equation}
    s^F_n(K^s_q)_{L_p} \leq Cn^{-s/d + (1/q - 1/p)_+},
\end{equation}
which is sharp when $p\leq q$ or $q \geq 2$. This expresses the fact that sampling the lowest frequencies is optimal in the linear regime $p\leq q$ or when $q \geq 2$.

\subsection{The non-linear regime $1\leq q < 2$ and $p > q$}
The case $1\leq q < 2$ and $p > q$ is more complicated and sampling the lowest frequencies is no longer optimal. To handle this case, we first observe that it suffices to consider the case $q = 1$. This is because for $q > 1$, H\"older's inequality implies that $\|f\|_{B^s_\infty(L_1)} \leq \|f\|_{B^s_\infty(L_q)}$, and since by assumption $s > 1-1/p$ it follows that $K^s_1$ is also a compact subset of $L_p$. Hence, proving \eqref{main-upper-bound-equation} for $q = 1$ also implies the same bound for $q > 1$.

Next, we use the  de la Vall\'ee Poussin kernel to construct the following multiscale decomposition of $f$. We define $f_0 := V_2f$ and for $k \geq 1$ we set $f_k := V_{2^{k+1}}f - V_{2^k}f$. We first observe that
\begin{equation}\label{decomposition-of-f-equation}
    \sum_{k=0}^\infty f_k = \lim_{r\rightarrow  \infty} \sum_{k=0}^{r-1} f_k = \lim_{r\rightarrow  \infty} V_{2^r}f = f
\end{equation}
with convergence in $L_p$, by \eqref{de-la-valee-bound-1-equation} and the density of trigonometric polynomials. Moreover, we verify the following properties of the $f_k$:
\begin{enumerate}
    \item $\hat{f}_k(\xi) = 0$ if $|\xi|_\infty < \lfloor2^{k-1}\rfloor$ or $|\xi|_\infty > 2^{k+1}$.
    \item $\hat{f}_k(\xi) = 0$ if $\hat{f}(\xi) = 0$ for all $\xi$.
    \item $\|f_k\|_{L_1} \leq C2^{-ks}$ for a constant $C$ independent of $k$ and $f\in K^s_1$.
\end{enumerate}
The first two properties above immediately follow by considering the Fourier multiplier corresponding to $V_{2^{k+1}}f - V_{2^k}f$, which vanishes outside of $\lfloor2^{k-1}\rfloor \leq |\xi|_\infty\leq 2^{k+1}$. The third property follows using \eqref{de-la-valee-bound-1-equation} and \eqref{besov-space-characterization} since $f\in K^s_1$ and thus
\begin{equation}\label{f-k-L1-bound-equation}
    \|f_k\|_{L_1} \leq \|f - V_{2^{k+1}}f\|_{L_1} + \|f - V_{2^k}f\|_{L_1} \leq CE_{2^{k-1}}(f)_1 \leq C2^{-ks}.
\end{equation}
At this point, let us give an overview of the remainder of the proof. Suppose that frequencies $\xi_1,...,\xi_n$ have been chosen, and $f\in K^s_1$ satisfies $\hat{f}(\xi_i) = 0$. The first two properties above then imply that $f_k$ is supported only on frequencies in the band
\begin{equation}
    B_k := \{\xi:~\lfloor2^{k-1}\rfloor \leq |\xi|_\infty\leq 2^{k+1}\},
\end{equation}
and also that $f_k$ vanishes at all of the $\xi_i$ which lie within this band, i.e., $\hat{f}_k(\xi_i) = 0$. Moreover, $f_k$ satisfies the $L_1$-bound given in the third property above. 

Our strategy will be to choose a set of frequencies $S_k\subset B_k$ of size $|S_k| = n_k$ for each frequency band $B_k$. The sizes $n_k$ are a parameter which we will optimize later. We then sample at the union of the sets $S_k$, i.e.,
\begin{equation}\label{total-collection-of-frequencies-equation}
    \{\xi_i\}_{i=1}^n = \bigcup_{k=0}^\infty S_k.
\end{equation}
The total number of frequencies chosen is bounded by
\begin{equation}\label{total-number-of-measurements-equation}
    n \leq \sum_{k=0}^\infty n_k. 
\end{equation}
Note that since the frequency bands $B_k$ overlap, it is possible for strict inequality to hold in \eqref{total-number-of-measurements-equation}, but this can only reduce the total number of measurements. By construction, it then follows that if $f\in K^s_1$ and $\hat{f}(\xi_i) = 0$ for $i=1,...,n$, then 
\begin{itemize}
    \item $f_k$ is supported on the frequency band $B_k$ and $\hat{f}_k(\xi) = 0$ for all $\xi\in S_k$, i.e., $f_k$ is supported on complement $B_k\backslash S_k$.
    \item $\|f_k\|_{L_1} \leq C2^{-ks}$.
\end{itemize} 

The frequencies $S_k$ will be chosen so that on the span of the complement $B_k\backslash S_k$ the $L_p$-norm is bounded by a small multiple of the $L_1$-norm, i.e., we want
\begin{equation}
    \|h\|_{L_p} \leq C(S_k)\|h\|_{L_1}
\end{equation}
for any 
\begin{equation}
h(x) = \sum_{\xi\in B_k\backslash S_k}a_\xi e^{i\xi\cdot x}
\end{equation}
with $C(S_k)$ as small as possible. 

This is where the techniques of compressive sensing are used. Specifically, this problem is closely related to restricted isometry property \cite{candes2005decoding} of Fourier matrices, which has been studied in \cite{candes2006near,rudelson2008sparse,cheraghchi2013restricted,haviv2017restricted}. The results in any of these papers would suffices to prove the upper bound \eqref{main-upper-bound-equation} with additional logarithmic factors. However, to obtain the sharpest possible bound we will use the following result from \cite{guedon2008majorizing}. We remark that this result is not strictly stronger than the results in \cite{rudelson2008sparse,cheraghchi2013restricted,haviv2017restricted}, but it is stronger in the regime which dominates the error in our problem.

\begin{theorem}[Theorem 3 in \cite{guedon2008majorizing}]\label{guedon-theorem}
    There exists an absolute constant $C$ with the following property. Let $\phi_1,...,\phi_N$ be a bounded orthonormal system, i.e., the $\{\phi_j\}_{j=1}^N$ are orthonormal in $L_2$ and $\|\phi_j\|_{L_\infty} \leq 1$ for $j=1,...,N$. Then for any $1 < k < N$ there exists a subset $S \subset \{1,...,N\}$ such that $|S| \leq n$ and for any scalars $a_j\in \mathbb{C}$ we have
    \begin{equation}\label{guedon-bound-equation}
        \left\|\sum_{j\in S^c} a_j\phi_j\right\|_{L_2}\leq C\mu(\log(\mu))^{5/2}\left\|\sum_{j\in S^c} a_j\phi_j\right\|_{L_1},
    \end{equation}
    where $\mu = \sqrt{N/n}\sqrt{\log{n}}$.
\end{theorem}
In Theorem \ref{guedon-theorem} the $L_p$ norms can be taken on any probability space, in particular on $\mathbb{T}^d$. We remark that the restriction $\|\phi_j\|_{L_\infty} \leq 1$ can be relaxed to $\|\phi_j\|_{L_\infty} \leq L$ at the expense of increasing $\mu$ proportionally. However, we will only apply this result to subsets of the Fourier system on the torus, and so will not need this additional generality. We remark also that the index set $S$ in Theorem \ref{guedon-theorem} can be uniformly randomly sub-sampled, and the resulting set will satisfy the bound \eqref{guedon-bound-equation} with high probability (see Theorem 3.10 in \cite{rudelson2008sparse} and Theorem 2 in \cite{guedon2008majorizing}). However, to bound the Fourier sampling numbers the existence of the set $S$ is sufficient.

We will utilize Theorem \ref{guedon-theorem} in the following way. Each frequency band $B_k$ is a bounded orthonormal system of size $|B_k| \leq 2^{d(k+1)}$. Supposing that $1 < n_k < |B_k|$, we may apply Theorem \ref{guedon-theorem}, we see that there exists a subset $S_k\subset B_k$ with $|S_k| \leq n_k$ such that for any 
\begin{equation}
h(x) = \sum_{\xi\in B_k\backslash S_k}a_\xi e^{i\xi\cdot x}
\end{equation}
we have
\begin{equation}
    \|h\|_{L_2} \leq C\mu_k(\log(\mu_k))^{5/2}\|h\|_{L_1}
\end{equation}
where $\mu_k = \sqrt{2^{d(k+1)}/n_k}\sqrt{\log{n_k}}$. In particular, since $f_k$ is supported on $B_k\backslash S_k$ it follows from the $L_1$-bound \eqref{f-k-L1-bound-equation} that
\begin{equation}
    \|f_k\|_{L_2} \leq C\mu_k(\log(\mu_k))^{5/2}\|f_k\|_{L_1} \leq C2^{-ks}\mu_k(\log(\mu_k))^{5/2}.
\end{equation}

To move from the $L_2$-norm error to general $L_p$-norm error, we observe that if $1 < p < 2$ a standard interpolation argument implies that
\begin{equation}\label{Lp-bound-less-2}
    \|f_k\|_{L_p} \leq \|f_k\|_{L_1}^{2/p - 1}\|f_k\|_{L_2}^{2 - 2/p} \leq C\left[\mu_k(\log(\mu_k))^{5/2}\right]^{2-2/p} \|f_k\|_{L_1} \leq C2^{-ks}\left[\mu_k(\log(\mu_k))^{5/2}\right]^{2-2/p},
\end{equation}
while if $p > 2$ the Nikolskii inequality \cite{nikol1951inequalities} implies that
\begin{equation}\label{Lp-bound-gtr-2}
    \|f_k\|_{L_p} \leq C2^{d(k+1)(1/2-1/p)}\|f_k\|_{L_2} \leq C2^{dk(1/2-1/p)}2^{-ks}\left[\mu_k(\log(\mu_k))^{5/2}\right],
\end{equation}
since $f_k$ is a trigonometric polynomial of degree at most $2^{k+1}$.

On the other hand, if $n_k = |B_k|$, i.e., all frequencies are in the set $S_k$, then obviously $f_k = 0$. Finally, if $n_k = 0$ we again use the Nikolskii inequality \cite{nikol1951inequalities} to get
\begin{equation}\label{tail-bound-nikolskii-equation}
    \|f_k\|_{L_p} \leq 2^{d(k+1)(1-1/p)}\|f\|_{L_1} \leq C2^{k(d(1-1/p)-s)}.
\end{equation}
Finally, we complete our analysis by optimizing the number of frequencies $n_k$ to select in each band. Fix an $\alpha > 0$ such that $(d+\alpha)(1-1/p) < s$ (this is possible since by assumption $s/d > 1-1/p$) and let $k_0 \geq 2$ be a parameter. We set
\begin{equation}
    n_k = \begin{cases}
        |B_k| & k \leq k_0\\
        \lceil2^{dk_0-\alpha(k-k_0)}\rceil & k_0 < k < k_0(1 + \frac{d}{\alpha})\\
        0 & k \geq k_0(1 + \frac{d}{\alpha}).
    \end{cases}
\end{equation}
Observe that in the middle case, $1 < n_k = \lceil2^{dk_0-\alpha(k-k_0)}\rceil \leq 2^{dk_0-\alpha(k-k_0) + 1}$, since the exponent $dk_0-\alpha(k-k_0)$ is greater than $0$ in this range. Using this, we bound $\mu_k$ in this range by
\begin{equation}\label{bound-on-mu-k}
    \mu_k \leq C2^{\frac{d+\alpha}{2}(k - k_0)}\sqrt{k_0 - \frac{\alpha}{d}(k-k_0)} \leq C2^{\frac{d+\alpha}{2}(k - k_0)}\sqrt{k_0},
\end{equation}
for a constant $C$ depending only on $d$ and $\alpha$.

We now estimate both the total number of frequencies selected and the maximum $L_p$-norm of any $f\in K^s_1$ which vanishes at the given frequencies as follows. As noted in \eqref{total-number-of-measurements-equation} the total number of frequencies is bounded by
\begin{equation}
    n\leq \sum_{k=0}^\infty n_k \leq 2^{d(k_0 + 1)} + 2^{dk_0+1}\sum_{k\geq k_0} 2^{-\alpha(k-k_0)} \leq C2^{dk_0}
\end{equation}
since $\alpha > 0$. 

Next, let $f\in K^s_1$ and suppose that $\hat{f}(\xi_i) = 0$ for the frequencies given in \eqref{total-collection-of-frequencies-equation}. Suppose first that $1 < p \leq 2$. Using \eqref{decomposition-of-f-equation}, \eqref{Lp-bound-less-2}, \eqref{bound-on-mu-k}, and \eqref{tail-bound-nikolskii-equation} we estimate
\begin{equation}
\begin{split}
    \|f\|_{L_p} &\leq \sum_{k=0}^\infty \|f_k\|_{L_p} \\
    &\leq C\left(2^{-sk_0}k_0^{1-\frac{1}{p}}\sum_{k_0 < k < k_0(1 + \frac{d}{\alpha})}2^{\left((d+\alpha)(1-\frac{1}{p})-s\right)(k-k_0)}\left(k-k_0 + \log(k_0)\right)^{5(1-\frac{1}{p})} + \sum_{k \geq k_0(1 + \frac{d}{\alpha})} 2^{k\left(d(1-\frac{1}{p})-s\right)}\right),
\end{split}
\end{equation}
since for $k \leq k_0$ all frequencies in $B_k$ and so $f_k = 0$ for such $k$. Since $s/d > 1-1/p$ and recalling the choice of $\alpha$, we have
\begin{equation}
    (d+\alpha)\left(1-\frac{1}{p}\right)-s < 0~\text{and}~d\left(1-\frac{1}{p}\right)-s < 0.
\end{equation}
This implies that the sums above converge and we obtain the bound
\begin{equation}
    \|f\|_{L_p} \leq C\left(2^{-sk_0}k_0^{1-\frac{1}{p}}(1+\log(k_0))^{5(1-\frac{1}{p})} + 2^{k_0(1+\frac{d}{\alpha})\left(d\left(1-\frac{1}{p}\right)-s\right)}\right).
\end{equation}
Finally, we observe that since $(d+\alpha)(1-1/p) < s$, it follows that
\begin{equation}\label{exponent-bound-derivation-equation}
    \left(1+\frac{d}{\alpha}\right)\left(d\left(1-\frac{1}{p}\right)-s\right) = \frac{d}{\alpha}(\alpha + d)\left(1-\frac{1}{p}\right) - \left(1+\frac{d}{\alpha}\right)s < -s.
\end{equation}
Thus, we obtain the bound (recalling that $k_0 \geq 2$)
\begin{equation}
    \|f\|_{L_p} \leq C2^{-sk_0}k_0^{1-\frac{1}{p}}\log(k_0)^{5(1-\frac{1}{p})}.
\end{equation}

On the other hand, when $p > 2$, we use \eqref{Lp-bound-gtr-2} instead of \eqref{Lp-bound-less-2} to estimate
\begin{equation}
\begin{split}
    \|f\|_{L_p} &\leq \sum_{k=0}^\infty \|f_k\|_{L_p} \\
    &\leq C\left(2^{-k_0\left(s - d(\frac{1}{2}-\frac{1}{p})\right)}\sqrt{k_0}\sum_{k_0 < k < k_0(1 + \frac{d}{\alpha})}2^{\left(\frac{d+\alpha}{2}-s+d(\frac{1}{2}-\frac{1}{p})\right)(k-k_0)}\left(k-k_0 + \log(k_0)\right)^{\frac{5}{2}} + \sum_{k \geq k_0(1 + \frac{d}{\alpha})} 2^{k\left(d(1-\frac{1}{p})-s\right)}\right),
\end{split}
\end{equation}
Since $s/d > 1-1/p$ it follows as before using \eqref{exponent-bound-derivation-equation} that the second sum above is bounded by $C2^{-sk_0}$. For this first sum, we observe that 
\begin{equation}
    \frac{d+\alpha}{2}-s+d\left(\frac{1}{2}-\frac{1}{p}\right) = (d+\alpha)\left(1 - \frac{1}{p}\right) - s - \alpha\left(\frac{1}{2} - \frac{1}{p}\right) < 0,
\end{equation}
since $(d+\alpha)(1-1/p) < s$ and $\alpha > 0$. This means that the first sum again converges and we get the bound
\begin{equation}
    \|f\|_{L_p} \leq C2^{-k_0\left(s - d(\frac{1}{2}-\frac{1}{p})\right)}\sqrt{k_0}\log(k_0)^{\frac{5}{2}} + C2^{-k_0} \leq C2^{-k_0\left(s - d(\frac{1}{2}-\frac{1}{p})\right)}\sqrt{k_0}\log(k_0)^{\frac{5}{2}}.
\end{equation}

Since $n\leq C2^{dk_0}$, this proves \eqref{main-upper-bound-equation} for all $n$ of the form $C2^{dk_0}$ with $k_0 \geq 2$, and thus by monotonicity for all $n \geq 4$ (adjusting the constant appropriately).

\section{Lower bounds}\label{lower-bounds-section}
In this section, we prove Theorem \ref{lower-bound-theorem}. Before presenting the proof, let us give a bit of context and related results.

The proof of Theorem \ref{lower-bound-theorem} is closely related to lower bounds on the size of Fourier matrices satisfying the restricted isometry property \cite{candes2005decoding,candes2006near,rudelson2008sparse}, and lower bounds on the problem considered in \cite{guedon2008majorizing,talagrand1998selecting}. Specifically, consider the measure space $X = \mathbb{C}^N$ with the empirical probability measure, i.e., for $x\in X$ we define
\begin{equation}
    \|x\|_{\ell_p} := \left(\frac{1}{N} \sum_{j=1}^n |x_j|^p\right)^{1/p}.
\end{equation}
For a given $1\leq n \leq N$, we wish to find a subspace $V_n\subset X$ of co-dimension $n$ on which the $\ell_1$ and $\ell_2$-norms are as comparable as possible, i.e., such that
\begin{equation}\label{l1-l2-comparison-equation}
    \|x\|_{\ell_1} \leq \|x\|_{\ell_2} \leq C\|x\|_{\ell_1}
\end{equation}
for all $x\in V_n$, with the constant $C:=C(V_n)$ as small as possible. We denote the best possible comparison that can be obtained for a subspace of co-dimension $n$ by
\begin{equation}\label{discrete-gelfand-width-definition-equation}
    C(n) := \inf_{V_n}~C(V_n).
\end{equation} 

Of course, this is simply a reformulation of the problem of estimating the Gelfand widths of the $\ell_1$ unit ball with respect to $\ell_2$ in $\mathbb{C}^N$. The connection with the restricted isometry property arises since if $V_n$ is taken to be the kernel of an $n\times N$ matrix $A$ which satisfies
\begin{equation}
    (1-\delta)\|x\|_2 \leq \|Ax\|_{2}\leq (1+\delta)\|x\|_2
\end{equation}
for all $s$-sparse vectors $x$, i.e., $x$ has at most $s$ non-zero entries, and $\delta = 1/4$, then $V_n$ satisfies the estimate \eqref{l1-l2-comparison-equation} with constant (see \cite{cohen2009compressed} or \cite{lorentz1996constructive}, Chapter 14)
$$C\lesssim \sqrt{N/s}.$$ 
Thus, constructing matrices satisfying the restricted isometry property \cite{candes2005decoding} gives a method for upper bounding the Gelfand widths of the $\ell_1$ unit ball with respect to $\ell_2$. This was essentially the same technique used by Garnaev and Gluskin \cite{garnaev1984widths} who obtained the tight estimate (see also \cite{lorentz1996constructive}, Chapter 14)
\begin{equation}\label{garnaev-gluskin-lower-bound}
    C(n) \eqsim \sqrt{\frac{N}{n}\left(1 + \log\frac{N}{n}\right)}.
\end{equation}

As in our investigation of `restricted' versions of the Gelfand widths, it is also of interest to restrict the allowed subspaces $V_n$ on which we seek an estimate of the form \eqref{l1-l2-comparison-equation}. One of the most important problems of this type arises when we restrict $V_n$ to be spanned by the characters of an abelian group. 

Specifically, let $(G,\cdot)$ be an abelian group of size $N$ and identify $\mathbb{C}^N$ with the space of complex-valued functions on $G$. The characters of $G$ are group homomorphisms from $G$ to $\mathbb{C}^*$ (the multiplicative group of non-zero complex numbers), i.e., they are functions $\chi:G\rightarrow \mathbb{C}\backslash \{0\}$ which satisfy
\begin{equation}
    \chi(g_1 \cdot g_2) = \chi(g_1)\chi(g_2).
\end{equation}
The characters form an abelian group under the action of point-wise multiplication, which we denote by $G^*$.
The identity element of $G^*$ is the function which is identically equal to $1$, which we denote by $\chi_0$.
The following properties of $G^*$ which we will use are well-known:
\begin{itemize}
    \item $|\chi(g)| = 1$ for all $\chi\in G^*$ and $g\in G$.
    \item The inverse of $\chi\in G^*$ is given by $\chi^{-1}(g) = \overline{\chi(g)}$ for all $g\in G$.
    \item Summing over the group, we have the relation
    \begin{equation}
    \frac{1}{|G|}\sum_{g\in G}\chi(g) = \begin{cases}
        1 & \chi = \chi_0\\
        0 & \chi \neq \chi_0.
    \end{cases}
    \end{equation}
    \item For any $\chi_1,\chi_2\in G^*$, we have
    \begin{equation}
        \langle \chi_1,\chi_2\rangle := \frac{1}{|G|}\sum_{g\in G}\chi_1(g)\overline{\chi_2(g)} = \begin{cases}
        1 & \chi_1 = \chi_2\\
        0 & \chi_1 \neq \chi_2,
    \end{cases}
    \end{equation}
    i.e., the characters are orthonormal with respect to the normalized $\ell_2$-inner product on $\mathbb{C}^N$.
    \item $G^* \cong G$. In particular, there are exactly $N$ characters and they form an orthonormal basis of $\mathbb{C}^N$.
\end{itemize}
As a prominant example, the (discrete) Fourier modes are the characters of the cyclic group $\mathbb{Z}/N\mathbb{Z}$.

It is an important problem in harmonic analysis to find subspaces spanned by a subset of the characters $\chi_1,...,\chi_N$ on which the $\ell_1$ and $\ell_2$-norms are comparable as in \eqref{l1-l2-comparison-equation}. More generally one can consider orthonormal systems of bounded functions, which are typically called bounded orthonormal systems, but the characters of a finite abelian group are one of the most common examples. Indeed, this is exactly the problem considered in \cite{talagrand1998selecting,guedon2008majorizing}. Of course, the lower bound contained in \eqref{garnaev-gluskin-lower-bound} also applies to this problem, but it is of interest to know whether a better lower bound holds when we restrict our subspace to be spanned by the characters of $G$, or more generally by a subset of a bounded orthonormal system.

Such an improved lower bound was obtained by Bourgain for the Walsh-Paley system, i.e., the characters of the group $G = (\mathbb{Z}/2\mathbb{Z})^k$ (where $N = 2^k$). Specifically, letting $C_{WP}(n)$ denote the same quantity as in \eqref{discrete-gelfand-width-definition-equation}, but with the infimum restricted to subspaces spanned by $N-n$ Walsh-Paley functions, i.e.,
\begin{equation}
    V_n = \text{span}\{\chi\in S\}~\text{for $S\subset G^*$ with $|S| = N-n$},
\end{equation}
Bourgain showed that for any $0 < \alpha < 1$, we have the lower bound
\begin{equation}\label{bourgain-lower-bound}
    C_{WP}(n) \geq c(\alpha)\sqrt{\frac{N\log(N)}{n}}~~\text{if $\sqrt{N} \leq n \leq (1-\alpha)N$}.
\end{equation}
The details of this argument can be found in the Appendix of \cite{guedon2008majorizing} or in \cite{foucart2013invitation}, Chapter 12. This bound improves upon \eqref{garnaev-gluskin-lower-bound} when $n$ is close to $N$, since the $\log(N/n)$ is replaced by $\log(N)$.

Unfortunately, Bourgain's improved lower bound only holds for the Walsh-Paley system. In order to prove Theorem \ref{lower-bound-theorem}, we need a lower bound specifically for the Fourier system, i.e., the characters of the group $G = \mathbb{Z}/N\mathbb{Z}$. Such a bound is actually claimed in a remark at the end of \cite{guedon2008majorizing}, but unfortunately their argument is not correct. Instead, we will need to modify the argument of Bourgain to obtain a lower bound specifically for the Fourier system (which incidentally gives a new lower bound on the size of Fourier matrices satisfying the restricted isometry property). The lower bound we prove is worse than \eqref{bourgain-lower-bound}, but applies to the characters of any abelian group $G$.

The first key ingredient we need is the following lemma, which generalizes the key piece of Bourgain's argument.
\begin{lemma}\label{bourgain-fundamental-lemma}
    Let $0 < \delta  < 1$ and $G$ be an abelian group. For any subset $\Lambda\subset G$ of size $|\Lambda| \geq 2$, there exists a set $S\subset G$ and an element $h\in G$ satisfying the following properties:
    \begin{enumerate}
        \item Letting $|S| = n$, we have
        \begin{equation}
            2^n \geq \frac{1}{2}\min\left\{\delta\frac{\log |G|}{\log (|G|/|\Lambda|)}, \left(\frac{|G|^{1-\delta}}{16}\right)^{2/3}\right\}.
        \end{equation}
        \item For any subset $T\subset S$, we have
        \begin{equation}\label{product-in-lambda-equation}
            h\cdot\prod_{g\in T}g\in \Lambda.
        \end{equation}
        \item If $T_1,T_2\subset S$ and $T_1\neq T_2$, then
        \begin{equation}\label{products-not-equation-equation}
            \prod_{g\in T_1}g\neq \prod_{g\in T_2}g.
        \end{equation}
    \end{enumerate}
\end{lemma}
Before giving the proof of Lemma \ref{bourgain-fundamental-lemma}, which is postponed to the end of this section, we will use it to prove the following result.
\begin{proposition}\label{riesz-product-of-characters-proposition}
    Fix an $0 < \alpha < 1$. Let $G$ be an abelian group and $\Lambda\subset G^*$ a set of characters with $|G|^{3/4} \leq |\Lambda| \leq (1-\alpha)|G|$. Then there exists a function $f:G\rightarrow \mathbb{C}$ in the span of the complement of $\Lambda$, i.e.,
    \begin{equation}
        f = \sum_{\chi\in \Lambda^c} a_\chi\chi,
    \end{equation}
    such that
    \begin{equation}
        \|f\|_{\ell_1} = 1~~\text{and}~~~\|f\|_{\ell_p} \geq c(\alpha) \left(\frac{|\Lambda|\log |G|}{|G|}\right)^{\gamma_p}\text{for $1 < p \leq 2$,}
    \end{equation}
    where $c(\alpha)$ is a constant depending only upon $\alpha$, and the exponent $\gamma_p$ is given by
    \begin{equation}\label{exponent-proposiiton-equation}
        \gamma_p = \max\left\{0,\frac{\log\left(\frac{\pi}{2^{1+1/p}}\right)}{\log{2}}\right\}.
    \end{equation}
\end{proposition}
Proposition \ref{riesz-product-of-characters-proposition} is only informative if $\gamma_p > 0$, which occurs when $p > 1/(\log(\pi)/\log(2) - 1) \approx 1.535$. 

Setting $p = 2$, Proposition \ref{riesz-product-of-characters-proposition} implies a weaker version of Bourgain's bound \eqref{bourgain-lower-bound} which holds for all groups $G$.  Namely, letting $C_G$ denote the same quantity as in \eqref{discrete-gelfand-width-definition-equation}, but with the infimum restricted to subspaces spanned by $N-n$ characters of $G$, for any $0 < \alpha < 1$ we have the lower bound
\begin{equation}\label{bourgain-lower-bound}
    C_{G}(n) \geq c(\alpha)\left(\frac{N\log(N)}{n}\right)^{\gamma_2}~~\text{if $N^{3/4} \leq n \leq (1-\alpha)N$}.
\end{equation}
Here the exponent $\gamma_2\approx 0.151$. Despite the small exponent, this result improves upon the Gelfand lower bound \eqref{garnaev-gluskin-lower-bound} when $n$ is very close to $N$.

Let us now use Proposition \ref{riesz-product-of-characters-proposition} to prove Theorem \ref{lower-bound-theorem}. Suppose that $n$ is given and let $k$ be the smallest integer such that $(2^{k} + 1)^d \geq 2n$. Note that $2^k \eqsim n^{1/d}$. Consider the group 
\begin{equation}
    G = \left(\mathbb{Z}/2^{k+2}\mathbb{Z}\right)^d = \{-2^{k+1},...,2^{k+1}\}^d.
\end{equation}
Here we have identified the group $G$ with the set of $d$-tuples of integers in the range $[-2^{k+1},2^{k+1}]$. The characters of this group can then be identified with Fourier modes in the same range, i.e.,
\begin{equation}
    G^* = \{\chi_\xi := e^{2\pi i\xi\cdot x}:~\xi\in \mathbb{Z}^d~\text{and $|\xi|_\infty\leq 2^{k+1}$}\}.
\end{equation}
We identify each such character with the Fourier mode $\phi_\xi := e^{i\xi\cdot x}$ on the torus $\mathbb{T}^d$. Suppose that we consider a trigonometric polynomial of degree at most $2^k$, i.e.,
\begin{equation}
    P(x) = \sum_{|\xi|_\infty \leq 2^k} a_\xi\phi_\xi.
\end{equation}
The Marcinkiewicz-Zygmund inequality (see for instance Chapter 10 of \cite{zygmund2002trigonometric}) implies that the $L_p$-norms of $P$ are equivalent to the $\ell_p$-norms of the corresponding linear combination of characters on $G$, i.e.,
\begin{equation}\label{marcin-zygmund-inequality}
    \left\|\sum_{|\xi|_\infty \leq 2^k} a_\xi\phi_\xi\right\|_{L_p} \eqsim \left\|\sum_{|\xi|_\infty \leq 2^k} a_\xi\chi_\xi\right\|_{\ell_p}.
\end{equation}
Finally, suppose that $n$ frequencies $\xi_1,...,\xi_n\in \mathbb{Z}^d$ are sampled. Consider the set
\begin{equation}
    \Lambda := \{\xi\in \mathbb{Z}^d:~2^k < |\xi|_\infty \leq 2^{k+1}\}\cup\{\xi_i:~|\xi_i|_\infty\leq 2^k\}.
\end{equation}
viewed as a subset of $G^*$. Since $(2^k + 1)^d \geq 2n$, the second set above contains at most half of the frequencies satisfying $|\xi|_\infty \leq 2^k$. Thus,
\begin{equation}
    c_d|G| \leq \Lambda \leq (1-c_d)|G|
\end{equation}
for a dimension dependent constant $c_d < 1/2$. Applying Proposition \ref{riesz-product-of-characters-proposition} to $G$ and $\Lambda$, we obtain a linear combination of characters satisfying
\begin{equation}
    \left\|\sum_{\substack{|\xi|_\infty \leq 2^k\\ \xi\notin\{\xi_1,...,\xi_n\}}} a_\xi\chi_\xi\right\|_{\ell_1} = 1,~~\text{and}~~\left\|\sum_{\substack{|\xi|_\infty \leq 2^k\\ \xi\notin\{\xi_1,...,\xi_n\}}} a_\xi\chi_\xi\right\|_{\ell_p} \geq c\log(|G|)^{\gamma_p} \geq c\log(n)^{\gamma_p},
\end{equation}
where the constant $c$ depends only upon $d$. Applying \eqref{marcin-zygmund-inequality} and using Bernstein's inequality, we see that the corresponding linear combination of frequencies satisfies
\begin{equation}
    \left\|\sum_{\substack{|\xi|_\infty \leq 2^k\\ \xi\notin\{\xi_1,...,\xi_n\}}} a_\xi\phi_\xi\right\|_{L_p} \geq c\log(|G|)^{\gamma_p} \geq c\log(n)^{\gamma_p}
\end{equation}
and
\begin{equation}
        \left\|\sum_{\substack{|\xi|_\infty \leq 2^k\\ \xi\notin\{\xi_1,...,\xi_n\}}} a_\xi\phi_\xi\right\|_{B^s_\infty(L_1)} \leq C2^{sk}\left\|\sum_{\substack{|\xi|_\infty \leq 2^k\\ \xi\notin\{\xi_1,...,\xi_n\}}} a_\xi\phi_\xi\right\|_{L_1} \leq C2^{ks} \leq Cn^{s/d}.
\end{equation}
Hence the ratio between the Besov norm $B^s_\infty(L_1)$ and the $L_p$-norm exceeds $cn^{-s/d}\log(n)^{\gamma_p}$ for a function vanishing at the frequencies $\xi_1,...,\xi_n$. Since these frequencies were arbitrary, it follows that
\begin{equation}
    s^F_n(K^s_1)_{L_p} \geq cn^{-s/d}\log(n)^{\gamma_p},
\end{equation}
for a constant $c > 0$ depending only on $d,s$ and $p$, but not on $f$ or $n$. This completes the proof of Theorem \ref{lower-bound-theorem}.

We conclude this section by giving the proofs of Proposition \ref{riesz-product-of-characters-proposition} and Lemma \ref{bourgain-fundamental-lemma}.

\begin{proof}[Proof of Proposition \ref{riesz-product-of-characters-proposition}]
    We apply Lemma \ref{bourgain-fundamental-lemma} to the group $G^*$ and the set $\Lambda^c$ with $\delta = 1/2$. This gives a set $S = \{\chi_1,...,\chi_n\}\subset G^*$ and a character $\chi_h\in G^*$ satisfying the properties of Lemma \ref{bourgain-fundamental-lemma} with
    \begin{equation}
        2^n \geq C\min\left\{\frac{\log |G|}{-\log (1 - |\Lambda|/|G|)}, |G|^{1/3}\right\}.
    \end{equation}
    Since $|\Lambda| \leq (1-\alpha)|G|$ it follows that $-\log (1 - |\Lambda|/|G|) \leq C(\alpha) (|\Lambda|/|G|)$, and so
    \begin{equation}
        2^n \geq c(\alpha)\min\left\{\frac{|G|\log |G|}{|\Lambda|}, |G|^{1/3}\right\}.
    \end{equation}
    Next, since $|\Lambda| \geq |G|^{3/4}$ it follows that $|G|\log |G|/|\Lambda| \leq C|G|^{1/3}$, and so
    \begin{equation}\label{lower-bound-on-n-equation-proposition}
        2^n \geq c(\alpha)\frac{|\Lambda|\log |G|}{|G|}.
    \end{equation}
    for a constant $c(\alpha)$ depending only upon $\alpha$.

    One idea for constructing $f$ based on the this set $S$ is to consider the Riesz product (indeed this was proposed in the appendix of \cite{guedon2008majorizing})
    \begin{equation}
        \rho_S := \chi_h\prod_{j=1}^n(1 + \chi_j).
    \end{equation}
    From the second property in Lemma 1, we see that $\rho_S$ is spanned by the characters in $\Lambda^c$, while the third property in Lemma 1 shows that $\rho_S$ is in fact the sum of $2^n$ distinct characters in $\Lambda^c$. From this it follows that $\|\rho_S\|^2_{\ell_2} = 2^n$, and that $\|\rho_S\|_{\ell_\infty} = 2^n$. However, unfortunately we cannot conclude that $\|\rho_S\|_{\ell_1} = 1$ (as claimed in the appendix of \cite{guedon2008majorizing}) since although
    \begin{equation}
        \frac{1}{G}\sum_{g\in G} \prod_{j=1}^n(1 + \chi_j(g)) = 1,
    \end{equation}
    the terms in this sum are not real and positive (as $\chi_j(g)$ is in general a \textit{complex} number of norm $1$). In the special case where all characters are real-valued, i.e., the case where $G = (\mathbb{Z}/2\mathbb{Z})^k$, this argument does work and gives Bourgain's lower bound \eqref{bourgain-lower-bound}.

    For general groups $G$, where at least some characters will be complex-valued, we proceed as follows. Observe that for any set of complex numbers $z_1,...,z_n\in \mathbb{C}$ with $|z_j| = 1$, the previous remarks also apply to the modified Riesz product
    \begin{equation}
        \rho_S(z_1,...,z_n) := \rho_S := \chi_h\prod_{j=1}^n(1 + z_j\chi_j).
    \end{equation}
    Specifically, for any $z_1,...,z_n$, $\rho_S(z_1,...,z_n)$ is spanned by the characters in $\Lambda^c$, and
    \begin{equation}
        \|\rho_S(z_1,...,z_n)\|^2_{\ell^2} = \|\rho_S(z_1,...,z_n)\|_{\ell_\infty} = 2^n.
    \end{equation}
    By a simple interpolation argument, this also implies that
    \begin{equation}\label{interpolation-lower-bound}
        \|\rho_S(z_1,...,z_n)\|_{\ell^p} \geq 2^{n(1-1/p)}~~\text{for $1\leq p \leq 2$}.
    \end{equation}
    Next, let us consider the average $\ell_1$-norm over uniformly random choices of $z_j$. We calculate (writing $z_j = e^{2\pi it_j}$)
    \begin{equation}
    \begin{split}
        \frac{1}{(2\pi)^n}\int_{0}^{2\pi}\cdots \int_{0}^{2\pi}&\|\rho_S(e^{2\pi it_1},...,e^{2\pi it_n})\|_{\ell_1}dt_1\cdots dt_n\\
        &= \frac{1}{(2\pi)^n}\int_{0}^{2\pi}\cdots \int_{0}^{2\pi}\frac{1}{|G|}\sum_{g\in G}\prod_{j=1}^n|1 + e^{2\pi it_j}\chi_j(g)|dt_1\cdots dt_n\\
        &= \frac{1}{|G|}\sum_{g\in G}\frac{1}{(2\pi)^n}\prod_{j=1}^n\left(\int_{0}^{2\pi}|1 + e^{2\pi it}\chi_j(g)|dt\right).
    \end{split}
    \end{equation}
    Now, since $|\chi_j(g)| = 1$, each of the integrals appearing the above product are equal. Specifically, for every $j = 1,...,n$ and $g\in G$ we have
    \begin{equation}
    \begin{split}
        \int_{0}^{2\pi}|1 + e^{2\pi it}\chi_j(g)|dt = \int_0^{2\pi} |1 + e^{2\pi it}|dt &= \sqrt{2}\int_{0}^{2\pi} \sqrt{1 + \cos(t)}dt\\
        &= 4\sqrt{2}\left[\sqrt{1 - \cos(t)}\right]_0^{\pi} = 8.
    \end{split}
    \end{equation}
    Hence, we see that the average $\ell_1$-norm is equal to
    \begin{equation}
        \frac{1}{(2\pi)^n}\int_{0}^{2\pi}\cdots \int_{0}^{2\pi}\|\rho_S(e^{2\pi it_1},...,e^{2\pi it_n})\|_{\ell_1}dt_1\cdots dt_n = \left(\frac{4}{\pi}\right)^n,
    \end{equation}
    and thus there exists a choice $z_1^*,...,z_n^*\in \mathbb{C}$ with $|z_j| = 1$ such that
    \begin{equation}\label{l1-upper-bound-modified-riesz-product}
        \|\rho_S(z^*_1,...,z^*_n)\|_{\ell_1} \leq (4/\pi)^n
    \end{equation}
    
    We set 
    \begin{equation}
        f := \frac{\rho_S(z^*_1,...,z^*_n)}{\|\rho_S(z^*_1,...,z^*_n)\|_{\ell_1}}.
    \end{equation}
    Obviously, $\|f\|_{\ell_1} = 1$, so that $\|f\|_{\ell_p} \geq 1$ for $p \geq 1$. In addition, \eqref{interpolation-lower-bound} combined with the upper bound \eqref{l1-upper-bound-modified-riesz-product} implies that
    \begin{equation}
        \|f\|_{\ell_p} \geq \left(\frac{\pi}{2^{1 + 1/p}}\right)^n.
    \end{equation}
    Finally, we use \eqref{lower-bound-on-n-equation-proposition} to conclude that
    \begin{equation}
        \|f\|_{\ell_p} \geq (2^n)^{\gamma_p}\geq \left(c(\alpha)\frac{|\Lambda|\log |G|}{|G|}\right)^{\gamma_p},
    \end{equation}
    where $\gamma_p = \log(\pi/2^{1 + 1/p})/\log(2)$. This statement is only useful when $\gamma_p > 0$, and since $\gamma_p$ is bounded for $p \leq 2$, we get
    \begin{equation}
        \|f\|_{\ell_p} \geq c(\alpha)\left(\frac{|\Lambda|\log |G|}{|G|}\right)^{\gamma_p},
    \end{equation}
    which completes the proof.
\end{proof}
\begin{proof}[Proof of Lemma \ref{bourgain-fundamental-lemma}]
    The proof largely follows the argument given in the appendix of \cite{guedon2008majorizing}, with a few key modifications. For an element $g\in G$ and a subset $S\subset G$, we write
    \begin{equation}
        g\cdot S := \{gh:~h\in S\}
    \end{equation}
    for the product of $g$ with all the elements of $S$.
    
    We iteratively construct a sequence of sets $S_0,S_1,...$ and $\Lambda_0,\Lambda_1,...$ as follows.
Set $S_0 = \emptyset$ and $\Lambda_0 = \Lambda$. Suppose that $S_j$ and $\Lambda_j$ have been constructed and that $\Lambda_j \neq \emptyset$. Define an `excluded' set via
    \begin{equation}
        E_j := \left\{\prod_{g\in S_j} g^{p_g}:~p_g\in \{0,\pm1\}~\text{for each $g$}\right\},
    \end{equation}
    and observe that $|E_j| \leq 3^{|S_j|}$. If $E_j = G$, or if $\Lambda_j\cap g^{-1}\cdot\Lambda_j = \emptyset$ for every $g\in G\backslash E_j$, we have completed the construction and set $S := S_j$ and choose any $h\in \Lambda_j$. Otherwise, we select
    \begin{equation}
        g_j\in \arg\max_{g\in G\backslash E_j} |\Lambda_j\cap g^{-1}\cdot\Lambda_j|,
    \end{equation}
    and set $S_{j+1} = S_j\cup\{g_j\}$ and $\Lambda_{j+1} = \Lambda_j\cap g^{-1}_j\cdot\Lambda_j\neq \emptyset$. Note that $\Lambda_{j+1}\subset \Lambda_j$.

    Next, we verify that at each step the second and third properties claimed in the lemma are satisfied by $S_j$ and any $h\in \Lambda_j$. This is obvious when $j=0$ since $S_0 = \emptyset$. Suppose now that these properties hold at step $j$. 
    
    To verify that the second property holds at step $j+1$, let $T\subset S_{j+1}$ and $h\in \Lambda_{j+1}$. If $T \subset S_j$, then since $h\in \Lambda_{j+1}\subset \Lambda_j$, we see by induction that \eqref{product-in-lambda-equation} holds. Otherwise, $g_j\in T$ and so
    \begin{equation}\label{shifted-product-lambda-j-equation}
        h\cdot\prod_{g\in T}g = (g_j\cdot h)\cdot \prod_{g\in T\backslash \{g_j\}}g.
    \end{equation}
    Since $h\in \Lambda_{j+1} \subset g_j^{-1}\cdot\Lambda_j$, it follows that $g_jh\in \Lambda_j$. Hence, by induction we see that the right hand side of \eqref{shifted-product-lambda-j-equation} is in $\Lambda$, and thus \eqref{product-in-lambda-equation} holds. 
    
    To verify that the third property holds, let $T_1,T_2\subset S_{j+1}$ and $T_1\neq T_2$. If $g_j \notin T_1\Delta T_2$, then either $g_j$ doesn't appear or it can be canceled from both sides of \eqref{products-not-equation-equation}, so that \eqref{products-not-equation-equation} holds by induction. If $g_j\in T_1\Delta T_2$, say $g_j\in T_1$ and $g_j\notin T_2$, then by rearranging \eqref{products-not-equation-equation} is equivalent to
    \begin{equation}
        g_j\neq \left(\prod_{g\in T_2}g\right)\left(\prod_{g\in T_1\backslash \{g_j\}}g\right)^{-1},
    \end{equation}
    and so \eqref{products-not-equation-equation} holds because $g_j\notin E_j$.

    Finally, to complete the proof, we estimate how many steps this iterative construction will take before completion. Observe that the construction stops at step $j = n-1$ only if either $3^{n-1} \geq |E_{n-1}| = |G|$ or if $|\Lambda_n| = 0$. Thus, we will certainly reach step $n$ provided that we can prove that $|\Lambda_j| \geq 2\cdot 3^j$ for $j=0,...,n$. We will prove by induction on $n$ that this holds provided that
    \begin{equation}\label{assumption-on-n-induction-equation}
        2^n \leq \delta\frac{\log |G|}{\log (|G|/|\Lambda|)}~~\text{and}~~2^n \leq \left(\frac{|G|^{1-\delta}}{16}\right)^{2/3}.
    \end{equation} 
    Since $|S_n| = n$, and the construction will reach step $n$ as long as \eqref{assumption-on-n-induction-equation} is satisfies, this verifies the first property in the lemma and completes the proof.
    
    The base case $n = 0$ follows since $|\Lambda| \geq 2$. Now let $n \geq 1$ be such that \eqref{assumption-on-n-induction-equation} holds and $|\Lambda_{j}| \geq 2\cdot 3^{j}$ for $j=0,...,n-1$. The key to the inductive step is the following lower bound on the size of $|\Lambda_{j+1}|$ in terms of $|\Lambda_{j}|$. Observe that
    \begin{equation}
        \sum_{g\in G} |\Lambda_j\cap g^{-1}\cdot\Lambda_j| = \sum_{g\in G} \sum_{h\in \Lambda_j} 1_{gh\in \Lambda_j} = \sum_{h\in \Lambda_j} \sum_{g\in G}1_{gh\in \Lambda_j} = |\Lambda_j|^2.
    \end{equation}
    Since we also obviously have $|\Lambda_j\cap g^{-1}\cdot\Lambda_j| \leq |\Lambda_j|$ it follows that
    \begin{equation}
        \sum_{g\in G\backslash E_j} |\Lambda_j\cap g^{-1}\cdot\Lambda_j| \geq |\Lambda_j|(|\Lambda_j| - |E_j|) \geq |\Lambda_j|(|\Lambda_j| - 3^j),
    \end{equation}
    since $|S_j| = j$ and so $|E_j| \leq 3^j$. It follows that there eixsts a $g \in G\backslash E_j$ such that
    \begin{equation}
        |\Lambda_j\cap g^{-1}\cdot\Lambda_j| \geq \frac{|\Lambda_j|(|\Lambda_j| - 3^j)}{|G| - |E_j|} \geq \frac{|\Lambda_j|(|\Lambda_j| - 3^j)}{|G|},
    \end{equation}
    and thus
    \begin{equation}\label{lambda-j-size-lower-bound}
        |\Lambda_{j+1}| \geq |\Lambda_j|\frac{(|\Lambda_j| - 3^j)}{|G|}.
    \end{equation}
    Rewriting the equation \eqref{lambda-j-size-lower-bound}, we obtain
    \begin{equation}
        \frac{|\Lambda_{j}|}{|G|} \geq \left(\frac{|\Lambda_{j-1}|}{|G|}\right)^2\left(1 - \frac{3^{j-1}}{|\Lambda_{j-1}|}\right),
    \end{equation}
    and taking logarithms, we obtain
    \begin{equation}
        \log\left(\frac{|\Lambda_{j}|}{|G|}\right) \geq 2\log\left(\frac{|\Lambda_{j-1}|}{|G|}\right) + \log\left(1 - \frac{3^{j-1}}{|\Lambda_{j-1}|}\right).
    \end{equation}
    If $|\Lambda_{j-1}| \geq 2\cdot 3^{j-1}$, then we may use the inequality $\log(1 - x) \geq -2x$ for $0 < x < 1/2$ (the logarithms here are taken to base $2$) to obtain the bound
    \begin{equation}
        \log\left(\frac{|\Lambda_{j}|}{|G|}\right) \geq 2\log\left(\frac{|\Lambda_{j-1}|}{|G|}\right) - 2\frac{3^{j-1}}{|\Lambda_{j-1}|}.
    \end{equation}
    Iteratively applying this inequality and using the inductive assumption that $|\Lambda_{j-1}| \geq 2\cdot 3^{j-1}$ for $j=1,...,n$, we get
    \begin{equation}
    \begin{split}
        \log\left(\frac{|\Lambda_{n}|}{|G|}\right) \geq 2^{n}\log\left(\frac{|\Lambda|}{|G|}\right) - 2\sum_{k=0}^{n-1}\frac{2^{n-k-1}3^k}{|\Lambda_k|} &\geq 2^{n}\log\left(\frac{|\Lambda|}{|G|}\right) - 2\frac{1}{|\Lambda_{n-1}|}\sum_{k=0}^{n-1}2^{n-k-1}3^k\\
        &\geq 2^{n}\log\left(\frac{|\Lambda|}{|G|}\right) - 2\frac{3^{n-1}}{|\Lambda_{n-1}|}\sum_{k=0}^{\infty}\left(\frac{2}{3}\right)^k\\
        &\geq 2^{n}\log\left(\frac{|\Lambda|}{|G|}\right) - 2\frac{3^{n}}{|\Lambda_{n-1}|}\\
        & \geq 2^{n}\log\left(\frac{|\Lambda|}{|G|}\right) - 3.
    \end{split}
    \end{equation}
    This implies that 
    \begin{equation}
        |\Lambda_n| \geq \frac{|G|}{8}\left(\frac{|\Lambda|}{|G|}\right)^{2^n}.
    \end{equation}
    Using the first inequality in \eqref{assumption-on-n-induction-equation} this implies that
    \begin{equation}
        |\Lambda_n| \geq \frac{|G|^{1-\delta}}{8}.
    \end{equation}
    Finally, the second inequality in \eqref{assumption-on-n-induction-equation} implies that
    \begin{equation}
        2\cdot 3^n \leq 2\cdot (2^n)^{3/2} \leq \frac{|G|^{1-\delta}}{8}.
    \end{equation}
    Thus, $|\Lambda_n| \geq 2\cdot 3^n$ and this completes the inductive step.
\end{proof}

\section{Numerical experiments}\label{numerical-experiments-section}
Let us begin this section by discussing practical algorithms for the estimation of $f^*$ from linear measurements $m_i = \lambda_i(f)$. Assume that $f^*\in B^s_\infty(L_q)$. Note that we have weakened the assumption that $f^*\in K^s_q$, i.e., we do not assume an a priori bound on the norm $\|f\|_{B^s_\infty(L_q)}$. We consider obtaining an estimate $\tilde{f}$ from the measurements $m_i$ by solving the convex optimization problem
\begin{equation}\label{general-convex-optimization-problem}
    \tilde{f} = \arg\min_{\lambda_i(f) = m_i} \|f\|_{B^s_\infty(L_q)}.
\end{equation}
Observe that since $f^*$ itself certainly satisfies the measurements, we obtain $\|\tilde{f}\|_{B^s_\infty(L_q)} \leq \|f^*\|_{B^s_\infty(L_q)}$. This means that the estimation error satisfies
\begin{equation}
    \|\tilde{f} - f^*\|_{B^s_\infty(L_q)} \leq 2\|f^*\|_{B^s_\infty(L_q)}~\text{and}~\lambda_i(\tilde{f} - f^*) = 0,
\end{equation}
which implies that
\begin{equation}\label{fundamental-recovery-error-optimization-bound}
    \|\tilde{f} - f^*\|_{L_p} \leq 2\|f^*\|_{B^s_\infty(L_q)}\sup\{\|f\|_{L_q}:~f\in B^s_q~\text{and}~\lambda_i(f) = 1~\text{for}~i = 1,...,n\}.
\end{equation}
Hence, solving the optimization problem \ref{general-convex-optimization-problem} gives a near-optimal estimator for any set of linear measurements under the smoothness assumption $f^*\in B^s_\infty(L_q)$. Note also that this algorithm doesn't require knowledge of the norm $\|f^*\|_{B^s_\infty(L_q)}$. For example, if the measurements $\lambda_i$ are chosen to realize the Gelfand widths \eqref{gelfand-widths-asymptotics}, this implies the bound
\begin{equation}
    \|\tilde{f} - f^*\|_{L_p} \leq 2\|f^*\|_{B^s_\infty(L_q)}d^n(K^s_q)_{L_p},
\end{equation}
while if we choose nearly optimal Fourier coefficients via the hierarchical random sub-sampling scheme described after Theorem \ref{main-theorem}, we obtain the bound
\begin{equation}
    \|\tilde{f} - f^*\|_{L_p} \leq 2\|f^*\|_{B^s_\infty(L_q)}s^F_n(K^s_q)_{L_p}.
\end{equation}
Of course, this applies not just for the Besov spaces but for any centrally symmetric convex set $K$ with the Besov norms replaced by the the gauge norm of $K$, i.e., the norm for which $K$ is the unit ball, which can be defined by
\begin{equation}
    \|f\|_K := \inf\{t > 0:f/t\in K\}.
\end{equation}

Next, let us discuss the practical implications of the recovery bound proved in Theorem \ref{main-theorem}. We will do this by means of the following example. Let $d = 2$ and $\Omega\subset \mathbb{T}^2$ be a subset of finite perimeter. Suppose that the group truth $f^*$ is the characteristic function of $\Omega$, i.e., 
\begin{equation}
    f^*(x) = \begin{cases}
        1 & x\in \Omega\\
        0 & x\notin \Omega.
    \end{cases}
\end{equation}
Since $\Omega$ has finite perimeter, it follows that $f^*$ is of bounded variation and thus contained in (a scalar multiple of) the model class $K^1_1$, specifically, we have
\begin{equation}
    \|f^*\|_{B^1_\infty(L_1)} \leq \|f^*\|_{BV} < \infty.
\end{equation}
We consider the problem of recovering $f^*$ under this assumption from linear measurements. In particular, we will compare the results when the measurements are restricted to point samples (sampling numbers given in \eqref{sampling-numbers-bound-equation}) to the results when general linear measurements are allowed (Gelfand widths given in \eqref{gelfand-widths-asymptotics}). Theorem \ref{main-theorem} implies that restricting the measurements to be Fourier coefficients gives results which are, up to logarithmic factors, the same as the Gelfand widths. 

Suppose that the recovery error is $L_p$ for $1\leq p < 2$ (these are the values of $p$ for which we obtain a compact embedding $BV\subset\subset L_p$). Summarizing \eqref{gelfand-widths-asymptotics} and \eqref{sampling-numbers-bound-equation}, from $n$-measurements we obtain a recovery error
\begin{equation}\label{bound-eq-844}
    \|\tilde{f} - f^*\|_{L_p} \leq C\|f^*\|_{BV}\begin{cases}
        n^{-1/2} & \text{general linear measurements}\\
        n^{-1/2 + 1 - 1/p} & \text{point samples.}
    \end{cases}
\end{equation}
We wish to examine the practical implications of these differing rates. To do this, let us consider the following three sets
\begin{enumerate}
    \item $T := \{x\in \mathbb{T}^d:~1/4 \leq \hat{f}(x)\leq 3/4\}$
    \item $P := \{x\notin \Omega:~\hat{f}(x) > 3/4\}$
    \item $N := \{x\in \Omega:~\hat{f}(x) < 1/4\}$
\end{enumerate}
We remark that there is nothing special about our choice of $1/4$ and $3/4$ in the above definitions, we could have chosen $c$ and $1-c$ for any $0 < c < 1/2$. 

We can interpret $T$ as the set on which the estimate $\tilde{f}$ \textit{transitions} between the values $0$ and $1$ taken by $f^*$, $P$ as the set of \textit{false positives} where $\tilde{f}$ is close to $1$, but $f^*$ equals $0$, and $N$ as the set of \textit{false negatives} where $\tilde{f}$ is close to $0$, but $f^*$ is equal to $1$.

We observe that $T\cup P\cup N \subset \{x:~|\hat{f}(x) - f^*(x)| \geq 1/4\}$, and thus
\begin{equation}
    |T\cup P\cup N| \leq \left(\frac{\|\tilde{f} - f^*\|_{L_p}}{4}\right)^p.
\end{equation}
Combining this with the bound \eqref{bound-eq-844} and optimizing over $p$ (for point samples we choose $p = 1$, while for general linear measurements we take $p\rightarrow 2$), we obtain the bound
\begin{equation}\label{error-set-bound-equation-865}
    |T\cup P\cup N| \leq C\begin{cases}
        (\|f^*\|_{BV}^{2}n^{-1})^\alpha & \text{general linear measurements (for any $\alpha < 1$)}\\
        \|f^*\|_{BV}n^{-1/2} & \text{point samples}.
    \end{cases}
\end{equation}
This shows that the rate at which $|T\cup P\cup N|$ decreases with $n$ is much faster when using general linear measurements vs point samples (essentially $n^{-1}$ vs $n^{-1/2}$).

From a practical standpoint this means the following: The ground truth $f^*$ has a sharp edge at the boundary of $\Omega$. The set $T$ is (roughly) the set where the estimate $\tilde{f}$ transitions from $0$ to $1$. Thus, if $|T|$ is small, then the estimate $\tilde{f}$ also exhibits a sharp edge. Further, $P$ is (roughly) the set enclosed by the edge of $\tilde{f}$ but not lying in $\Omega$, while $N$ is (roughly) the set lying in $\Omega$ by not enclosed by the edge of $\tilde{f}$. Hence, if both $|P|$ and $|N|$ are small, this means that the edge of $\tilde{f}$ is close to the true edge of $f^*$. So the measure $|T\cup P\cup N|$ is essentially a proxy for how accurately the edges of $f^*$ are captured by the estimate $\tilde{f}$.

If we return to the bound \eqref{error-set-bound-equation-865}, we see that when using point samples, the edges of $f^*$ can be captured only to accuracy $n^{-1/2}$. This certainly makes sense, since by sampling on a uniform grid with spacing $h$ we cannot expect to capture edges more accurately than the spacing $h$. On the other hand, using appropriate general linear measurements, we can capture edges with accuracy nearly $n^{-1}$! This greatly improved capturing of edges is precisely what is enabled by non-linear approximation (the recovery is non-linear), and is quantified by allowing a stronger error norm ($L_p$ for $p > 1$) without deterioration of the rate in \eqref{bound-eq-844}.

Theorem \ref{main-theorem} shows that the Fourier sampling numbers behave like the Gelfand widths up to logarithmic factors. This means that the same greatly improved capturing of edges is also possible if we restrict to measuring Fourier coefficients of $f^*$ instead of general linear measurements. We conclude this section by giving a numerical demonstration of this phenomenon.

As ground truth we take the function $f^*$ given by
\begin{equation}
    f^*(x,y) = -0.75 \cdot\begin{cases}
        1 & 1\leq x \leq 5 ~\text{and}~2\leq y\leq 4\\
        0 & \text{otherwise}
    \end{cases} -1.0\cdot \begin{cases}
        1 & 1\leq |x - 3| + |y - 4| \leq 1\\
        0 & \text{otherwise},
    \end{cases}
\end{equation}
extended $2\pi$-periodically.
This function is shown in Figure \ref{ground-truth-plot-figure}. The ground truth $f^*$ is a sum of two characteristic functions, and is therefore in $BV$. In addition, the exact Fourier coefficients of $f^*$ can be easily calculated.

\begin{figure}
    \includegraphics[scale = 1.0]{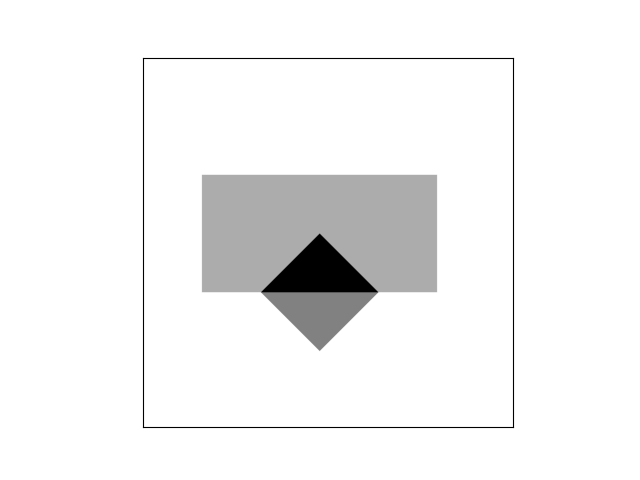}
    \caption{The ground truth function $f^*$ in our experiment.}\label{ground-truth-plot-figure}
\end{figure}

We compare different methods for recovering $f^*$ from its Fourier measurements. We want to emphasize that we are considering the recovery of the \textit{continuous} function $f^*$ from its \textit{exact continuous} Fourier coefficients. We do not discretize on a grid and compute discrete Fourier transforms, thus avoiding the well-known `inverse-crime' \cite{guerquin2011realistic}. We compare the recovery using the following different sampling and recovery methods (in each case we sample exactly $289$ Fourier coefficients):
\begin{enumerate}
    \item Sampling the lowest Fourier coefficients, i.e., sample the Fourier coefficients for frequencies $\xi$ on a $[-8,...,8]^2$ grid, and reconstructing by summing the Fourier series. The result is shown in Figure \ref{fourier-sum-plot-figure}. We see that the edges of $f^*$ have been smeared out significantly and the Gibbs phenomenon, i.e., spurious oscillations, are clearly evident.
    \item Sampling the lowest Fourier coefficients, i.e., sample the Fourier coefficients for frequencies $\xi$ on a $[-8,...,8]^2$ grid, and reconstructing using a smoothed Fourier sum (in our case a tensor product de la Vall\'ee Poussin sum). The result is shown in Figure \ref{de-la-vallee-sum-plot-figure}. The Gibbs phenomenon is no longer present, but the edges of $f^*$ have been smeared out even more.
    \item Sampling the lowest Fourier coefficients, i.e., sample the Fourier coefficients for frequencies $\xi$ on a $[-8,...,8]^2$ grid, and reconstructing by solving the optimization problem \eqref{general-convex-optimization-problem} for the model class $BV$, i.e., solving
    \begin{equation}\label{BV-reconstruction-optimization}
        \tilde{f} := \arg\min_{\hat{f}(\xi_i) = \hat{f^*}(\xi_i)} \|f\|_{BV}.
    \end{equation}
    The result is shown in Table \ref{lowest-frequencies-bv-norm-min-plot-figure}. We see that $\tilde{f}$ has sharp edges since it is obtained via a $BV$-norm minimization, however, the edges do not line up with the edges of $f^*$ particularly accurately. In particular, corners are rounded out significantly and point where four different values touch is significantly altered.
    \item Sampling Fourier coefficients via the hierarchical random sub-sampling scheme described following Theorem \ref{main-theorem}, and reconstruction by solving the optimization problem \eqref{BV-reconstruction-optimization}. The parameters $\alpha$ and $k_0$ have been chosen to ensure that exactly $289$ frequencies are sampled. The result is shown in Table \ref{hierarchical-bv-norm-min-plot-figure}. Although the reconstruction is still not perfect, we see that the edges of $f^*$ are much more accurately captured..
    \item Sampling Fourier coefficients uniformly at random from a $[-1024,...,1024]^2$ grid (this is the same grid size that the functions are plotted on), and reconstructing by solving the optimization problem \eqref{BV-reconstruction-optimization}. This sampling scheme is meant to test the sampling methods proposed for the discrete recovery problem. The result is shown in Figure \ref{uniformly-random-bv-norm-min-plot-figure}. We can see that a uniform random sub-sampling completely fails for the continuous recovery problem.
\end{enumerate}

\begin{figure}
    \includegraphics[scale = 1.0]{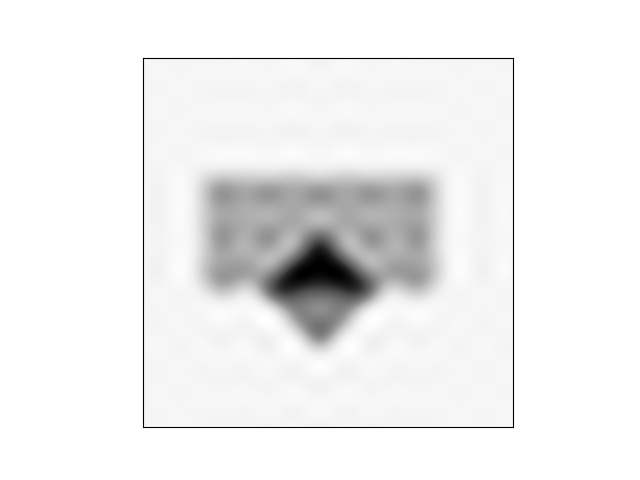}
    \caption{The reconstruction by summing the Fourier series on a $[-8,...,8]^2$ grid.}\label{fourier-sum-plot-figure}
\end{figure}
\begin{figure}
    \includegraphics[scale = 1.0]{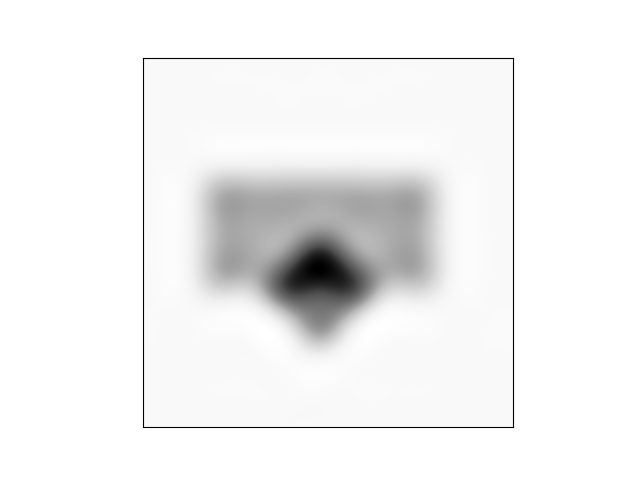}
    \caption{The reconstruction by taking the de la Vall\'ee Poussin sum on a $[-8,...,8]^2$ grid.}\label{de-la-vallee-sum-plot-figure}
\end{figure}
\begin{figure}
    \includegraphics[scale = 1.0]{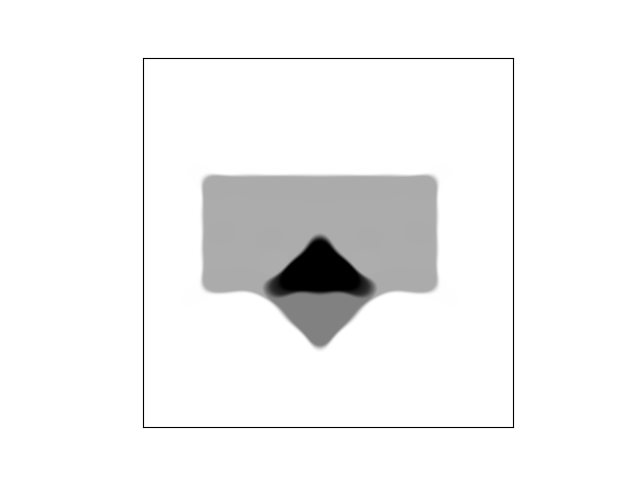}
    \caption{The reconstruction by sampling the Fourier coefficients on a $[-8,...,8]^2$ grid, and solving the $BV$-norm minimization problem \eqref{BV-reconstruction-optimization}.}\label{lowest-frequencies-bv-norm-min-plot-figure}
\end{figure}
\begin{figure}
    \includegraphics[scale = 1.0]{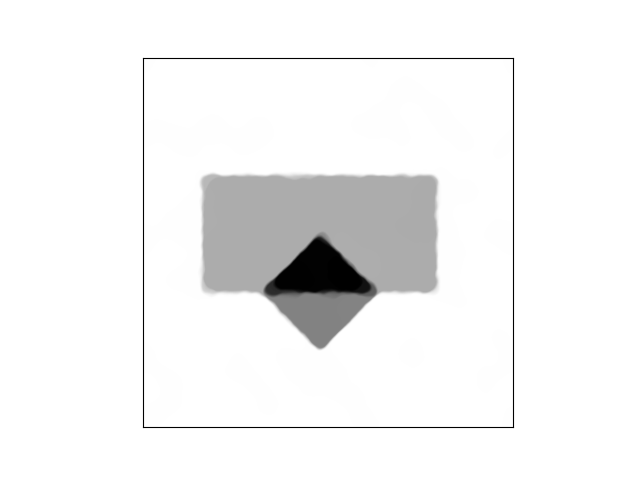}
    \caption{The reconstruction by sub-sampling the Fourier coefficients in a hierarchical manner, and solving the $BV$-norm minimization problem \eqref{BV-reconstruction-optimization}.}\label{hierarchical-bv-norm-min-plot-figure}
\end{figure}
\begin{figure}
    \includegraphics[scale = 1.0]{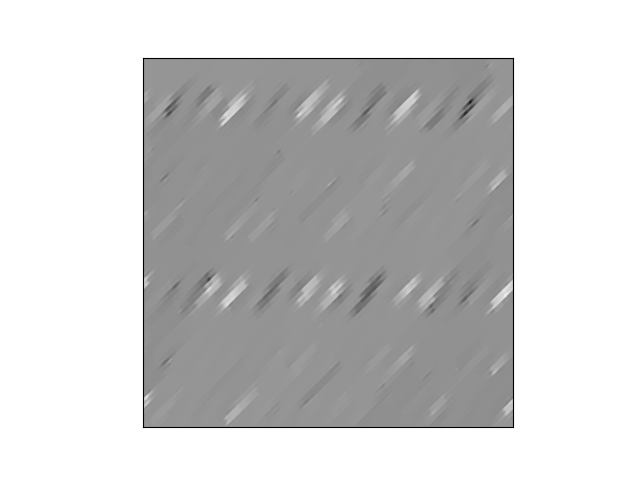}
    \caption{The reconstruction by uniformly sub-sampling the Fourier coefficients from a large grid, and solving the $BV$-norm minimization problem \eqref{BV-reconstruction-optimization}.}\label{uniformly-random-bv-norm-min-plot-figure}
\end{figure}
\begin{figure}
    \includegraphics[scale = 1.0]{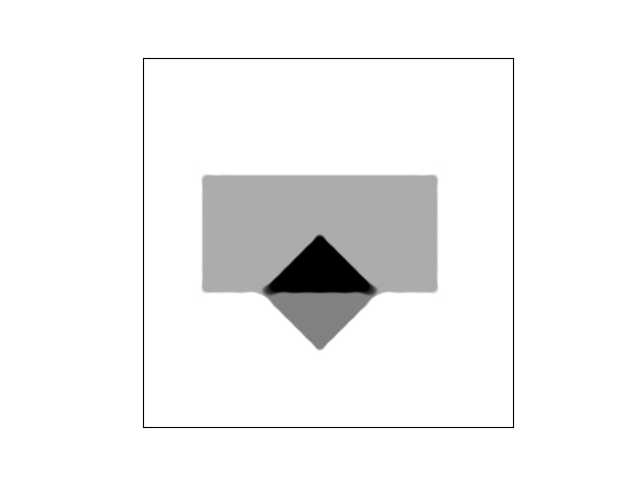}
    \caption{The reconstruction by sampling the Fourier coefficients on a $[-16,...,16]^2$ grid, and solving the $BV$-norm minimization problem \eqref{BV-reconstruction-optimization}.}\label{lowest-frequencies-bv-norm-min-plot-figure-high}
\end{figure}
\begin{figure}
    \includegraphics[scale = 0.75]{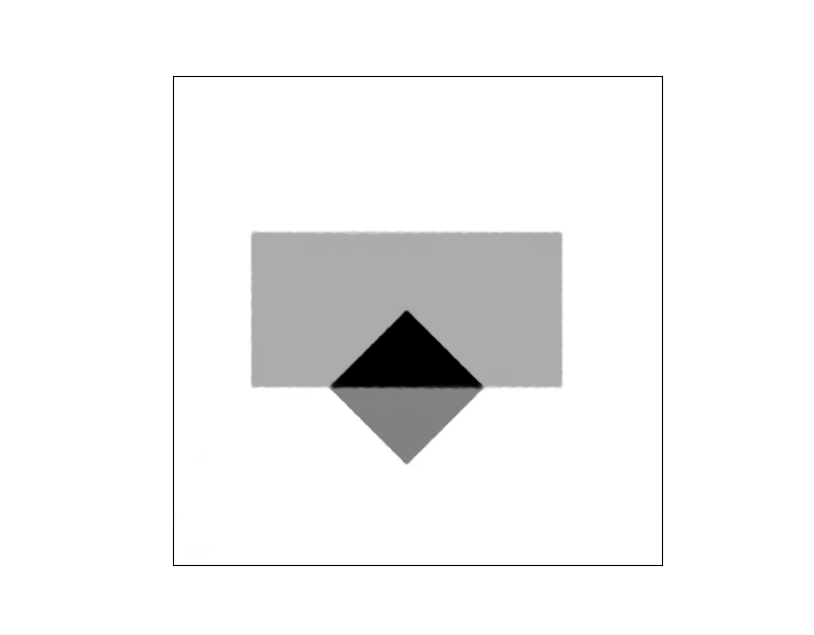}
    \caption{The reconstruction by sub-sampling the Fourier coefficients in a hierarchical manner (sampling a total of $1089$ frequencies, and solving the $BV$-norm minimization problem \eqref{BV-reconstruction-optimization}.}\label{hierarchical-bv-norm-min-plot-figure-high}
\end{figure}

To complete the description of our experiments, let us describe how the optimization problem \eqref{BV-reconstruction-optimization} was numerically solved. Specifically, we calculate a numerical solution $\tilde{f}_{num}$ which satisfies
\begin{equation}
    \hat{\tilde f}_{num}(\xi_i) = f^*(\xi_i)~\text{for $i=1,...,n$}~~~\text{and}~~~\|\tilde{f}_{num}\|_{BV} \leq C\inf_{\hat{f}(\xi_i) = \hat{f^*}(\xi_i)} \|f\|_{BV}
\end{equation}
for a constant $C$. Such a near optimizer is sufficient to guarantee the recovery error in \eqref{fundamental-recovery-error-optimization-bound} up to the same constant factor $C$. Let $f\in BV$ be any feasible point, i.e., any function satisfying $\hat{f}(\xi_i) = \hat{f^*}(\xi_i)$ for $i=1,...,n$. Choose $m$ large enough that $|\xi_i|_\infty \leq m/2$ for all $i$. Then
\begin{equation}
    f_m := V_mf\in \mathcal{T}_m^d
\end{equation}
is also a feasible point (recall the definition of the de la Vall\'ee Poussin sum \eqref{de-la-vallee-poisson-definition-equation}), and
\begin{equation}
    \|f_m\|_{BV} \leq C\|f\|_{BV}
\end{equation}
since the de la Vall\'ee Poussin kernel is bounded in $L_1$ and $BV$ is a translation invariant norm. Thus, it suffices to consider trigonometric polynomials of degree $m$ in order to obtain a near optimizer. On the space of trigonometric polynomials, the $BV$-norm is given by
\begin{equation}
    \|f_m\|_{BV} = \|\nabla f_m\|_{L_1}.
\end{equation}
The gradients of trigonometric polynomials can be calculated exactly (and are again trigonometric polynomials of the same degree), and the $L_1$-norm can be discretized up to a constant factor by computing the discrete $\ell_1$-norm on an oversampled grid, i.e., on a grid with spacing $2\pi/(4m+1)$, via the Marcinkiewicz-Zygmund inequality (see for instance Chapter 10 of \cite{zygmund2002trigonometric}). In this way, the entire problem can be reduced to a discrete linear programming problem of the form
\begin{equation}
    \min_{\substack{x\in \mathbb{R}^M\\Ax = b}} \|Bx\|_{\ell_1}
\end{equation}
for appropriate matrices $A$ and $B$. We solve this latter problem using the well-known alternating method of multipliers (ADMM) \cite{eckstein1992douglas}. All implementation details and the code for running these experiments can be found at \href{https://github.com/jwsiegel2510/Fourier-sampling-experiments}{https://github.com/jwsiegel2510/Fourier-sampling-experiments}.

Finally, let us remark that the phenomenon observed in our experiments already occurs for only $289$ Fourier measurements. Since the \textit{rates} of edge recovery are better using hierarchical random sub-sampling, this phenomenon will only become more pronounced as the number of samples is increased. As an example, we show the results of the same experiment with $1089$ sampled frequencies in Figures \ref{lowest-frequencies-bv-norm-min-plot-figure-high} (sampling the lowest frequencies) and \ref{hierarchical-bv-norm-min-plot-figure-high} (hierarchical random subsampling). The reconstruction method in both cases is by solving the BV-norm minimization problem \eqref{BV-reconstruction-optimization}. We see that when sampling the lowest frequencies, the corners and triple interesection points are still distorted, while the hierarchical random sub-sampling scheme produces an essentially perfect reconstruction. 

\section{Conclusion}\label{conclusion-section}
Using tools from compressive sensing, we have estimated the Fourier sampling numbers for Besov spaces with error measured in $L_p$. This provides a theoretical foundation for the application of non-linear reconstruction and compressive sensing methods to imaging problems such as MRI. In particular, we have developed the \textit{numerical analysis} required to apply the techniques of compressive sensing, which typically concern the recovery of sparse discrete signals from discrete measurements, to the problem of recovering a continuous function from continuous Fourier measurements. The key conclusion is that by hierarchically sub-sampling the Fourier coefficients in an appropriate manner, we can recover edges in the ground truth much more accurately than by sampling the lowest Fourier modes.

An important open problem, which we do not address, is the development of an analogous theory for Radon measurements. Such a theory would explain how to optimally sample for CT imaging applications, analogous to our theory which models MRI measurements.

\section{Acknowledgements}
This work was supported by the National Science Foundation (DMS-2424305) as well as the MURI ONR grant N00014-20-1-2787. We would like to thank Ronald DeVore, Rahul Parhi, Robert Nowak, Guergana Petrova, Albert Cohen, Tino Ullrich, and Mario Ullrich for helpful discussions regarding this work.

\bibliographystyle{amsplain}
\bibliography{refs}
\end{document}